\documentclass[11pt,letterpaper]{article}

\usepackage[T1]{fontenc}
\usepackage{lmodern}

\usepackage[utf8]{inputenc} 
\usepackage[T1]{fontenc}    
\usepackage[hidelinks]{hyperref}       
\usepackage{url}            
\usepackage{booktabs}       
\usepackage{amsfonts}       
\usepackage{nicefrac}       
\usepackage{microtype}      

\usepackage{amsthm}
\usepackage{amsmath}
\usepackage{algcompatible}
\usepackage{algorithm}

\usepackage[noend]{algpseudocode}

\usepackage{booktabs}
\usepackage{enumerate}
\usepackage{graphicx}
\usepackage{amssymb}
\usepackage{latexsym}
\usepackage{epstopdf}
\usepackage{color}
\usepackage{xcolor}
\usepackage{bbm}
\usepackage{needspace}
\usepackage{color, colortbl}
\usepackage[english]{babel}

\usepackage{caption}
\usepackage{subcaption}
\usepackage{filecontents}
\usepackage{mathtools}
\usepackage{multirow}

\usepackage[margin=1.0in]{geometry}

\newtheorem{theorem}{Theorem}

\newtheorem{lemma}{Lemma}
\newtheorem{definition}{Definition}
\newtheorem{remark}{Remark}
\newtheorem{assumption}{Assumption}
\newtheorem{problem}{Problem}

\newcommand{\tp}{\top}
\newcommand{\Proj}{\text{Proj}}
\newcommand{\tr}{\textbf{tr}}
\newcommand{\gr}{\text{gr}}
\newcommand{\dist}{\text{dist}}

\title{{\LARGE \bf Learning Lyapunov Functions for Piecewise Affine Systems with Neural Network Controllers}}
\author{Shaoru Chen, Mahyar Fazlyab, Manfred Morari, George J. Pappas, Victor M. Preciado
	\thanks{The authors are with the Department of Electrical and Systems Engineering, University of Pennsylvania. Email: \{srchen, mahyarfa, morari, pappasg, preciado\}@seas.upenn.edu.}} 

\date{}

\begin{document}
\pagestyle{plain}
\maketitle

\begin{abstract}

We propose a learning-based method for Lyapunov stability analysis of piecewise affine dynamical systems in feedback with piecewise affine neural network controllers. The proposed method consists of an iterative interaction between a learner and a verifier, where in each iteration, the learner uses a collection of samples of the closed-loop system to propose a Lyapunov function candidate as the solution to a convex program. The learner then queries the verifier, which solves a mixed-integer program to either validate the proposed Lyapunov function candidate or reject it with a counterexample, i.e., a state where the stability condition fails. This counterexample is then added to the sample set of the learner to refine the set of Lyapunov function candidates. We design the learner and the verifier based on the analytic center cutting-plane method, in which the verifier acts as the cutting-plane oracle to refine the set of Lyapunov function candidates. We show that when the set of Lyapunov functions is full-dimensional in the parameter space, the overall procedure finds a Lyapunov function in a finite number of iterations. We demonstrate the utility of the proposed method in searching for quadratic and piecewise quadratic Lyapunov functions.

\end{abstract}


\section{Introduction}
Deep neural networks (DNNs) have achieved remarkable success in various challenging tasks across different fields such as image classification~\cite{krizhevsky2012imagenet}, machine translation~\cite{sutskever2014sequence}, speech recognition~\cite{hinton2012deep}, reinforcement learning, and control~\cite{mnih2015human}. The high performance of DNNs, however, does not come with guarantees as they are highly complex and, therefore, hard to reason about. Moreover, it is relatively well-known that neural networks can be sensitive to input perturbations and adversarial attacks~\cite{kurakin2016adversarial, papernot2016limitations}. These drawbacks limit the application of DNNs in safety-critical systems, in which standard control strategies, although potentially inferior to DNNs in performance, do have performance guarantees. Therefore, it is critical to develop tools that can provide useful certificates of stability, safety, and robustness for DNN-driven systems. 

Stability analysis aims to show that a set of initial states of a dynamical system would stay around and potentially converge to an equilibrium point. 
Stability analysis of dynamical systems often relies on constructing Lyapunov functions that can find positive invariant sets, certify the stability of equilibrium points, and estimate their region of attraction. 
Finding non-conservative Lyapunov functions for generic nonlinear systems, however, is not an easy task and often requires substantial expertise, manual effort, and expensive numerics in the form of complex optimization problems. 
%
%
This task becomes even more challenging when neural networks (NNs) are used as stabilizing controllers. In this context, it is of great interest to automate the construction of Lyapunov functions for various classes of nonlinear systems \cite{abate2020formal} and, in particular, neural-network-controlled feedback systems. 

The complexity of finding Lyapunov functions for nonlinear systems has also motivated the use of data-driven methods, which aim at ``learning'' Lyapunov functions from observations or simulation traces. The intuition is that a Lyapunov function constructed from finitely many trajectories would hopefully be a valid Lyapunov function for \emph{all} possible trajectories. Although these methods do not require a model of the system to be available, they generally rely on a large number of simulations and often lack formal guarantees  despite the high confidence gained from many simulations.

\medskip

\textit{Contributions}: We propose a sample-efficient optimization-based method  to synthesize Lyapunov functions for discrete-time piecewise affine systems in feedback with ReLU neural network controllers. The proposed algorithm relies on a ``learner'', which proposes a Lyapunov function candidate using finitely many samples, and a ``verifier'', which either validates the Lyapunov function candidate or rejects it with a counterexample, i.e., a state where the stability condition fails. This counterexample is then added to the sample set of the learner to refine the set of Lyapunov function candidates. We design the learner and the verifier according to the analytic center cutting-plane method (ACCPM) from convex optimization, with which the algorithm can efficiently and exhaustively search over the considered Lyapunov function candidate class. Notably, the algorithm is guaranteed to terminate in finite steps if the set of Lyapunov function is full-dimensional in the parameter space. We demonstrate the application of our algorithm to search for quadratic and piecewise quadratic Lyapunov functions.

\subsection{Related work}

\emph{Output range analysis of NNs}: Driven by the need to verify the robustness of deep neural networks against adversarial attacks~\cite{kurakin2016adversarial}, a large body of work has been reported on bounding the output of DNNs for a given range of inputs~\cite{dutta2017output, lomuscio2017approach, wong2018provable, fazlyab2019safety, huang2017safety, raghunathan2018semidefinite, katz2017reluplex}. These methods can be categorized into exact or approximate depending on how they evaluate/approximate the nonlinear activation functions in the DNN. Of particular interest to this paper is the tight certification of ReLU NNs through formal verification techniques such as Satisfiability Modulo Theories (SMT) solvers~\cite{huang2017safety, katz2017reluplex, ehlers2017formal} or mixed-integer programming (MIP)~\cite{lomuscio2017approach, tjeng2017evaluating}. These approaches can verify the linear properties of the ReLU NN output given linear constraints on the network inputs. More scalable but less accurate methods that apply linear programming (LP)~\cite{wong2018provable} or semidefinite programming (SDP) relaxations~\cite{fazlyab2019safety, raghunathan2018semidefinite,fazlyab2019efficient,fazlyab2019probabilistic} of NNs to over approximate the output range can be found in the literature.  \\


\noindent \emph{Stability and reachability analysis of NN-controlled systems}: Despite their high performance, NN-controlled systems lack safety and stability guarantees. Motivated by this, several works have studied the challenging task of verifying NN-controlled systems.
In~\cite{hu2020reach}, built upon the framework of \cite{fazlyab2019safety}, the authors over-approximate the reachable set of NN-controlled LTI systems by bounding the nonlinear activation functions of the NN through quadratic constraints (QCs). In \cite{huang2019reachnn, dutta2019reachability}, the authors approximate NN functions through polynomial approximations to conduct a closed-loop reachability analysis. Ivanov et al.~\cite{ivanov2019verisig} transform a sigmoid-based NN into a hybrid system and propose the verification of the NN-controlled system through standard hybrid system verification tools. 

For stability analysis, linear matrix inequality (LMI)-based sufficient stability conditions are derived by abstracting the nonlinear activation function in NNs through QCs~\cite{kim2018standard} or by linear difference inclusions~\cite{tanaka1996approach, limanond1998neural}. Yin et al.~\cite{yin2020stability} use a QC abstraction of NNs to analyze uncertain plants with perturbations described by integral quadratic constraints (IQCs). The input-output stability of an NN-controlled system is considered in~\cite{jin2018stability} where QCs are constructed from the bounds of partial gradients of NN controllers. For ReLU NN controllers, the authors in~\cite{karg2020stability} exploit the fact that a ReLU NN induces a piecewise affine (PWA) function and analyze the closed-loop stability by identifying the linear dynamics around the equilibrium. However, the number of modes generated by the ReLU NN is exponential in the number of neurons, and can easily make general PWA system stability analysis tools intractable~\cite{biswas2005survey, ferrari2002analysis, johansson1997computation, rubagotti2016lyapunov}.\\

\noindent \emph{Stability analysis of PWA systems}: In this paper, we consider PWA systems in feedback with ReLU NN controllers. The closed-loop system is itself an autonomous PWA system since a ReLU NN is, in fact, a continuous PWA function~\cite{pascanu2013number}. For continuous-time PWA systems, LMI-based approaches to synthesize piecewise affine~\cite{johansson2003piecewise}, piecewise quadratic (PWQ)~\cite{johansson1997computation} and piecewise polynomial Lyapunov functions~\cite{prajna2003analysis} have been proposed. The adaptation of these Lyapunov function synthesis methods for discrete-time PWA systems is summarized in~\cite{biswas2005survey}. 

Although formulated as convex programs, these methods require parameterizing a Lyapunov function candidate on each mode, and identifying the pairwise one-step transition sets between the modes for discrete-time PWA systems. Since the number of partitions generated by a deep ReLU network can be extremely large~\cite{pascanu2013number}, the aforementioned methods involve solving large-scale convex problems. Taking PWQ Lyapunov function synthesis as an example, the approach proposed in~\cite{biswas2005survey} requires solving a large-scale SDP which is conservative and only provides a sufficient condition for stability due to the application of the $\mathcal{S}$-procedure~\cite{boyd1994linear}. In this paper, we propose an MIP-based iterative approach for quadratic and PWQ Lyapunov function synthesis for a discrete-time PWA system controlled by a ReLU NN controller. Our proposed method does not require identifying all the partitions of the ReLU network and is non-conservative. \\


\noindent \emph{Sample-based Lyapunov function synthesis}: Synthesizing Lyapunov functions for nonlinear dynamical systems is very difficult, in general, and often involves solving large, nonconvex optimization problems. To avoid solving hard optimization problems directly, several sample-based Lyapunov function synthesis methods have been proposed. Topcu et al.~\cite{topcu2007stability} use simulation data to help postulate the region of attraction of continuous-time polynomial systems and find Lyapunov function candidates for solving bilinear matrix inequalities. Boffi et al.~\cite{boffi2020learning} propose a data-driven method that can learn Lyapunov functions from system trajectory data with probabilistic guarantees. Closely related to our work is~\cite{kapinski2014simulation}, where an iterative approach formulates Lyapunov function candidates for continuous-time nonlinear systems from simulation traces and improves the result by the counterexample trace generated by a falsifier at each iteration. Our method differs from theirs in that we provide a finite-step termination guarantee for our iterative algorithm. \\

\noindent \emph{Counterexample guided inductive synthesis}: The iterative approach alternating between a learning module and a verification module to synthesize a certificate for control systems is known as the Counter-Example Guided Inductive Synthesis (CEGIS) framework proposed by~\cite{solar2006combinatorial, solar2008program} in the verification community. The application of CEGIS on Lyapunov function synthesis for continuous-time nonlinear autonomous systems can be found in~\cite{ahmed2020automated, abate2020formal} using SMT solvers for verification. However, the termination of the iterative procedures in these works is not guaranteed. Notably, Ravanbakhsh et al.~\cite{ravanbakhsh2019learning} apply the CEGIS framework to synthesize control Lyapunov functions for nonlinear continuous-time systems and provide finite-step termination guarantees for the iterative algorithm through careful design of the learner, which essentially implements the maximum volume ellipsoid cutting-plane method~\cite{Tarasov1988}. Although our method also provides termination guarantees through the (analytic center) cutting-plane method, the considered problem setting (stability verification of NN-controlled PWA system), and the construction of the learner and the verifier are highly different. \\

The rest of the paper is organized as follows. After introducing the preliminaries on Lyapunov functions in Section~\ref{sec:preliminaries}, we state our problem formulation as finding an estimate of the region of attraction of a NN-controlled PWA system in Section~\ref{sec:prob_statement}. We develop our method in Section~\ref{sec:ACCPM_Lyap} and discuss its extensions and variations in Section~\ref{sec:extensions}. Section~\ref{sec:numerical} demonstrates the application of the proposed method through numerical examples and Section~\ref{sec:conclusion} concludes the paper. 

\subsection{Notations}
We denote the set of real numbers by $\mathbb{R}$, the set of integers by $\mathbb{Z}$, the $n$-dimensional real vector space by $\mathbb{R}^n$, and the set of $n \times m$ real matrices by $\mathbb{R}^{n \times m}$. The standard inner product between two matrices $A, B \in \mathbb{R}^{n \times m}$ is given by $\langle A, B \rangle = \tr(A^\top B)$ and the Frobenius norm of a matrix $A \in \mathbb{R}^{n \times m}$ is given by $\lVert A \rVert_F = (\tr(A^\top A))^{1/2}$. We denote the set of $n \times n$ symmetric matrices by $\mathbb{S}^n$, and $\mathbb{S}^n_{+}$ (resp., $\mathbb{S}^n_{++}$) denotes the set of $n \times n$ positive semidefinite (resp., definite) matrices. For two vectors $x \in \mathbb{R}^{n_x}$ and $y \in \mathbb{R}^{n_y}$, $(x, y) \in \mathbb{R}^{n_x + n_y}$ represents their concatenation. Given a set $\mathcal{S} \subseteq \mathbb{R}^{n_x + n_y}, \Proj_x (\mathcal{S}) = \{ x \in \mathbb{R}^{n_x} \vert \exists y \in \mathbb{R}^{n_y} \text{ s.t. } (x, y) \in \mathcal{S}\}$ denotes the orthogonal projection of $\mathcal{S}$ onto the subspace $\mathbb{R}^{n_x}$. Given a point $x \in \mathbb{R}^{n_x}$ and a compact set $\mathcal{R} \subseteq \mathbb{R}^{n_x}$, define the projection of $x$ to the set $\mathcal{R}$ as $\Proj_\mathcal{R}(x) = \arg \min_{y \in \mathcal{R}} \dist(x, y)$ where $\dist(x, y) = \lVert x - y \rVert_2$ is the Euclidean norm of $x - y$. For a set $\mathcal{S} \subseteq \mathbb{R}^n$, we denote $\text{cl}(\mathcal{S})$ the closure of $\mathcal{S}$, $\text{int}(\mathcal{S})$ the set of all interior points in $\mathcal{S}$, and $\mathcal{S}_\infty = \{d \in \mathbb{R}^n \vert x + \alpha d \in \mathcal{S}, \forall x \in \mathcal{S}, \forall \alpha \in \mathbb{R}_+ \}$ the recession cone of $\mathcal{S}$.

\section{Preliminaries}
\label{sec:preliminaries}
\subsection{Stability of nonlinear autonomous systems}
Consider a discrete-time autonomous system 
\begin{equation} \label{eq:general_system}
x_+ = f(x),
\end{equation}
where $x \in \mathbb{R}^{n_x}$ is the state and $f: \mathcal{R} \mapsto \mathbb{R}^{n_x}$ is a nonlinear, continuous function with domain $\mathcal{R} \subseteq \mathbb{R}^{n_x}$. Without loss of generality, we assume $0 \in \text{int}(\mathcal{R})$ is an equilibrium of the system, i.e., $f(0) = 0$. Denote $x_k$ the state of system~\eqref{eq:general_system} at time $k$ and $x_0$ the initial state. To study the stability properties of autonomous systems at their equilibrium points, Lyapunov functions may be used. 

\begin{definition} \normalfont{(Lyapunov stability \cite[Chapter~13]{haddad2011nonlinear})} \label{def:Lyap_stability}
	The equilibrium point $x = 0$ of the autonomous system~\eqref{eq:general_system} is
	\begin{itemize}
		\item \textbf{Lyapunov stable} if for each $\epsilon >0$, there exists $\delta = \delta(\epsilon)$ such that if 
		$\lVert x_0 \rVert < \delta$, then $\lVert x_k \rVert < \epsilon, \forall k \geq 0$.
		\item \textbf{asymptotically stable}  if it is Lyapunov stable and there exists $\delta > 0$ such that if
		$\lVert x_0 \rVert < \delta$, then $\lim_{k\rightarrow \infty} \|x_k\| = 0$.
		\item \textbf{geometrically stable} if there exists $\delta >0, \alpha > 1, 0< \rho < 1$ such that if
		$\lVert x_0 \rVert < \delta$, then $\lVert x_k \rVert \leq \alpha \rho^k \lVert x_0 \rVert \quad \forall k \geq 0$.
	\end{itemize}
\end{definition}
For an asymptotically stable equilibrium point, the region of attraction (ROA) is the set of initial states from which the trajectories of the autonomous system converge to the equilibrium.
\begin{definition}[Region of attraction]
	The ROA of the autonomous system~\eqref{eq:general_system} is defined as
	\begin{equation*}
	\mathcal{O} = \{ x_0 \in \mathcal{R} \vert \underset{k \rightarrow \infty}{\lim} \|x_k\| = 0 \}.
	\end{equation*}
\end{definition}
%

\begin{definition}[Successor set]
	For the autonomous system~\eqref{eq:general_system}, we denote the successor set from a set $\mathcal{X}$ as $\text{Suc}(\mathcal{X}) = \{ y \in \mathbb{R}^{n_x} \vert \exists x \in \mathcal{X} \text{ s.t. } y = f(x)  \}$.
\end{definition}

\begin{definition}[Positive invariant set] \label{def:invariance}
	A set $\mathcal{X} \subseteq \mathcal{R}$ is said to be positive invariant for the autonomous system~\eqref{eq:general_system} if for all $x_0 \in \mathcal{X}$, we have $x_k \in \mathcal{X}, \forall k \geq 0$.
\end{definition}

\begin{theorem}\normalfont{(\cite[Chapter~13]{haddad2011nonlinear})}  \label{thm:Lyap}
	Consider the discrete-time nonlinear system~\eqref{eq:general_system} and assume there is a continuous function $V(x): \mathcal{R} \mapsto \mathbb{R}$ with domain $\mathcal{R}$ such that
	\begin{subequations}
		\label{eq:Lyap_conditions}
		\begin{align}
		\begin{split} \label{eq:general_Lyap_cond_1}
		& V(0) = 0 \text{ and } V(x) > 0, \forall x \in \mathcal{X} \setminus \{0\}
		\end{split} \\
		\begin{split} \label{eq:general_Lyap_cond_3}
		& V(f(x)) - V(x) \leq 0, \forall x \in \mathcal{X},
		\end{split}
		\end{align}
	\end{subequations}
	where the set $\mathcal{X}$ is the region of interest (ROI) satisfying $\mathcal{X} \subseteq \mathcal{R}$, $\text{Suc}(\mathcal{X})\subseteq \mathcal{R}$, and $0 \in \text{int}(\mathcal{X})$. Then the origin is Lyapunov stable. If, in addition, 
\begin{equation}\label{eq:general_Lyap_cond_2}
	V(f(x)) - V(x) < 0, \forall x \in \mathcal{X} \setminus \{ 0 \},
\end{equation}
	then the origin is asymptotically stable.
\end{theorem}

We call any $V(\cdot)$ satisfying the condition~\eqref{eq:general_Lyap_cond_1} a Lyapunov function \emph{candidate}. If additionally, $V(\cdot)$ satisfies the condition~\eqref{eq:general_Lyap_cond_3} or~\eqref{eq:general_Lyap_cond_2}, then $V(\cdot)$ is called a \emph{valid} Lyapunov function candidate. Since asymptotic stability is our primary focus in this paper, a Lyapunov function is any $V(\cdot)$ satisfying constraints~\eqref{eq:general_Lyap_cond_1} and~\eqref{eq:general_Lyap_cond_2} unless specified otherwise. 
Let $\alpha = \inf_{x \in \mathcal{R} \setminus \mathcal{X}} V(x)$. Then the set 
\begin{equation}
\mathcal{\tilde{O}} = \{ x \vert V(x) \leq \alpha\},
\end{equation}
is an inner approximation of the ROA and $\mathcal{\tilde{O}} \subset \mathcal{X}$ is the largest sublevel set of $V(x)$ that is contained in $\mathcal{X}$. 

\section{Problem Statement}
\label{sec:prob_statement}
\subsection{Plant model}
In this paper, we consider discrete-time piecewise affine (PWA) dynamical systems of the form
\begin{equation} \label{eq:hybrid_dyn}
x_{+} = \psi_i(x, u) = A_i x + B_i u + c_i, \  \forall x \in \mathcal{R}_i=\{ x \in \mathbb{R}^{n_x} \vert F_i x \leq h_i \}, 
\end{equation}
where $x \in \mathbb{R}^{n_x}$ is the state, $u \in \mathbb{R}^{n_u}$ is the control input, $x_{+} \in \mathbb{R}^{n_x}$ denotes the state at the next time instant, $i \in \mathcal{I} = \{ 1, 2, \cdots, N_m \}$ is the mode of the system, and $\{\mathcal{R}_i \}_{i \in \mathcal{I}}$ are polytopic partitions (i.e., $\mathcal{R}_i$ is a bounded polyhedron) of the domain $\mathcal{R} = \cup_{i=1}^{N_m} \mathcal{R}_i$. We assume that the partitions $\mathcal{R}_i$ satisfy $\text{int}(\mathcal{R}_i) \cap \text{int}(\mathcal{R}_j) = \emptyset$ and the PWA dynamics~\eqref{eq:hybrid_dyn} is well-posed, i.e., $\psi_i(x,u) = \psi_j(x, u) , \forall x \in \mathcal{R}_i \cap \mathcal{R}_j, \forall(i,j) \in \mathcal{I}^2$. We denote the PWA system~\eqref{eq:hybrid_dyn} compactly as $x_{+} = \psi(x, u)$. Without loss of generality, we assume the origin $(x, u) = 0$ is an equilibrium point of the PWA system, i.e., $\psi(0, 0) = 0$. As a special case, the PWA system~\eqref{eq:hybrid_dyn} reduces to a linear time-invariant (LTI) system when there is only one mode. In this case, we denote the LTI system as
\begin{equation} \label{eq:LTI_dyn}
x_+ = Ax + Bu, \quad x \in \mathcal{R}= \{ x \in \mathbb{R}^{n_x} \vert F x \leq h \}.
\end{equation}

\subsection{Neural network controller model}
In this paper we assume the PWA dynamical system \eqref{eq:hybrid_dyn} is driven by a multi-layer piecewise linear neural network controller $u = \pi(x)$, where $\pi: \mathbb{R}^{n_x} \rightarrow \mathbb{R}^{n_u}$ is described by 

\begin{equation} \label{eq:nn_controller}
\begin{aligned}
z_0 &= x\\
z_{\ell+1} &= \max(W_\ell z_\ell + b_\ell, 0), \quad \ell = 0, \cdots, L -1 \\
\pi(x) &= W_{L} z_{L} + b_L.
\end{aligned}
\end{equation}
Here, $z_0 = x \in \mathbb{R}^{n_0} \ (n_0 = n_x)$ is the input to the neural network, $z_{\ell+1} \in \mathbb{R}^{n_{\ell+1}}$ is the vector representing the output of the $(\ell+1)$-th hidden layer with $n_{\ell+1}$ neurons, $\pi(x) \in \mathbb{R}^{n_{L+1}} \ (n_{L+1} = n_u)$ is the output of the neural network, and $W_\ell \in \mathbb{R}^{n_{\ell + 1} \times n_\ell}, b_\ell \in \mathbb{R}^{n_{\ell + 1}}$ are the weight matrix and the bias vector of the $(\ell+1)$-th hidden layer. 

\subsection{Closed-loop stability analysis}
\label{sec:cl_system}
We denote the closed-loop dynamics of the PWA system~\eqref{eq:hybrid_dyn} in feedback with the ReLU network controller~\eqref{eq:nn_controller} as
\begin{equation}\label{eq:cl_dynamics}
x_+ = f_{cl}(x) := \psi(x, \pi(x)).
\end{equation}
Under the assumption $\pi(0) = 0$, the origin is an equilibrium of the closed-loop system, i.e., $f_{cl}(0) = 0$. Define the region of interest as a polytopic set given by
\begin{equation} \label{eq:ROI}
\mathcal{X} = \{ x\in\mathbb{R}^{n_x} \vert F_\mathcal{X} x \leq h_\mathcal{X}\}.
\end{equation}
We are interested in verifying the stability of the equilibrium point and finding an inner estimate of its ROA contained in $\mathcal{X}$, where the ROI $\mathcal{X}$ can be interpreted as a user-defined set that guides the search for an estimate of the ROA.

\begin{problem}\label{prob:ROA}
	Find an inner estimate of the ROA, $\tilde{\mathcal{O}} \subseteq \mathcal{X}$, for the closed-loop system~\eqref{eq:cl_dynamics}. 
\end{problem}

\begin{figure}[hbt!]
\centering
\includegraphics[width = 0.35 \linewidth]{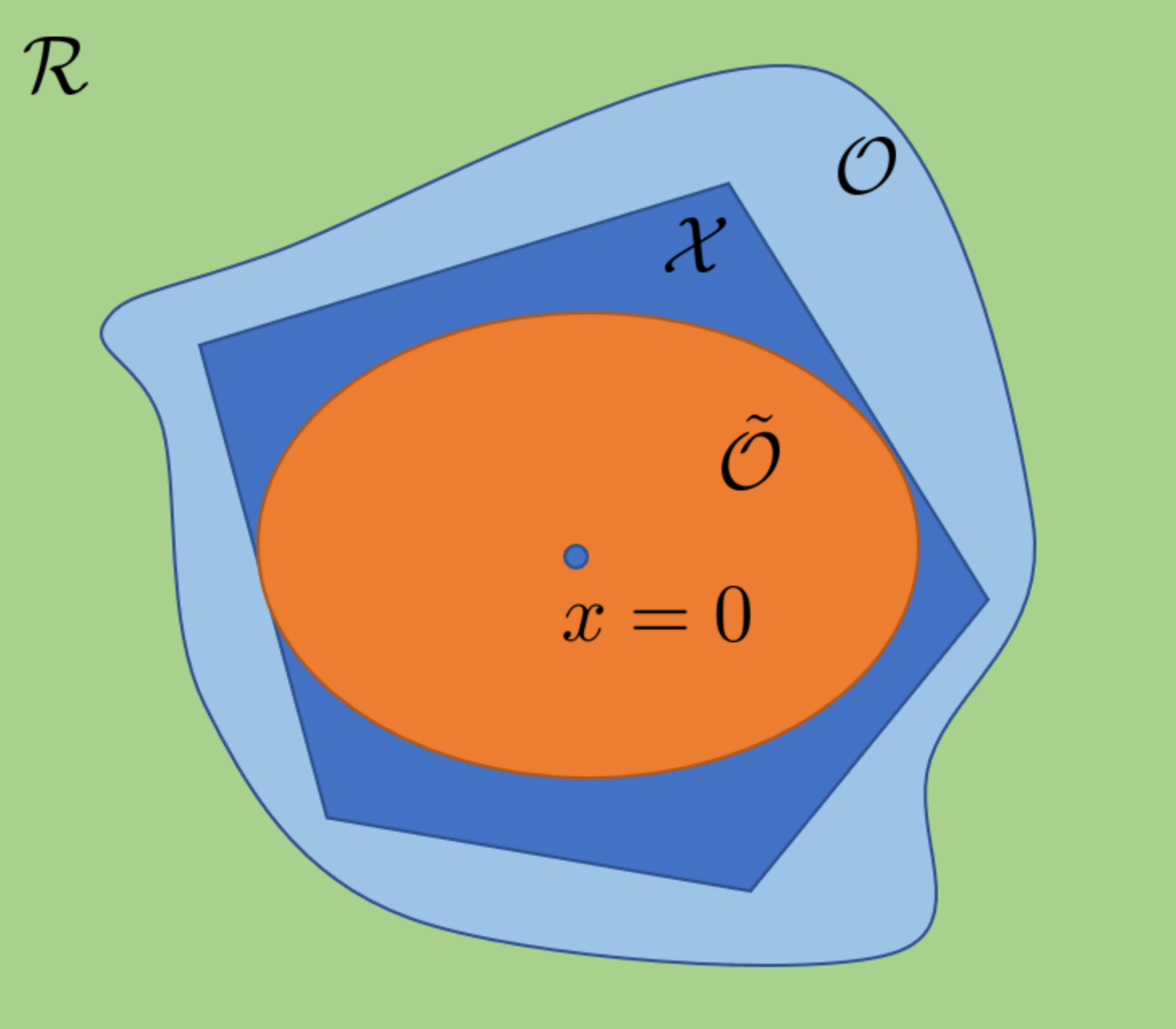}
\label{fig:prob_1_demo}
\caption{In Problem~\ref{prob:ROA}, the aim is to find an estimate of the ROA $\tilde{\mathcal{O}}$ inside the ROI $\mathcal{X}$. Note that the ROA $\mathcal{O}$ is an unknown, possibly nonconvex set and $\mathcal{X}$ is chosen agnostic to $\mathcal{O}$.}
\end{figure}

\section{Learning Lyapunov functions for the closed-loop system}
\label{sec:ACCPM_Lyap}

To solve Problem~\ref{prob:ROA}, we propose an iterative learning-based method to construct a valid Lyapunov function for the closed-loop system~\eqref{eq:cl_dynamics}.
In each iteration of this algorithm, a ``learner'' proposes a Lyapunov function candidate using a collection of samples of the closed-loop system. The learner then queries a ``verifier'' that either certifies that the proposed function is a valid Lyapunov function or rejects it with a counterexample, i.e., a state where the stability condition fails. This counterexample is then added to the training set of the learner. In the next iteration, the learner uses the new counterexample and a cutting-plane strategy to refine the set of Lyapunov function candidates. The learner repeats this process until termination when either it finds a valid Lyapunov function, or it verifies its non-existence in the given parameterized function space. In our framework, the learner solves a convex optimization problem to synthesize Lyapunov function candidates, and the verifier solves a mixed-integer program (MIP) to generate the counterexamples. We provide a finite-step termination guarantee for our algorithm through the analysis of cutting-plane methods from convex optimization. The proposed method is illustrated in Fig.~\ref{fig:learner_verifier}.

\begin{figure}[htb!]
	\centering
	\includegraphics[width = 0.8 \textwidth]{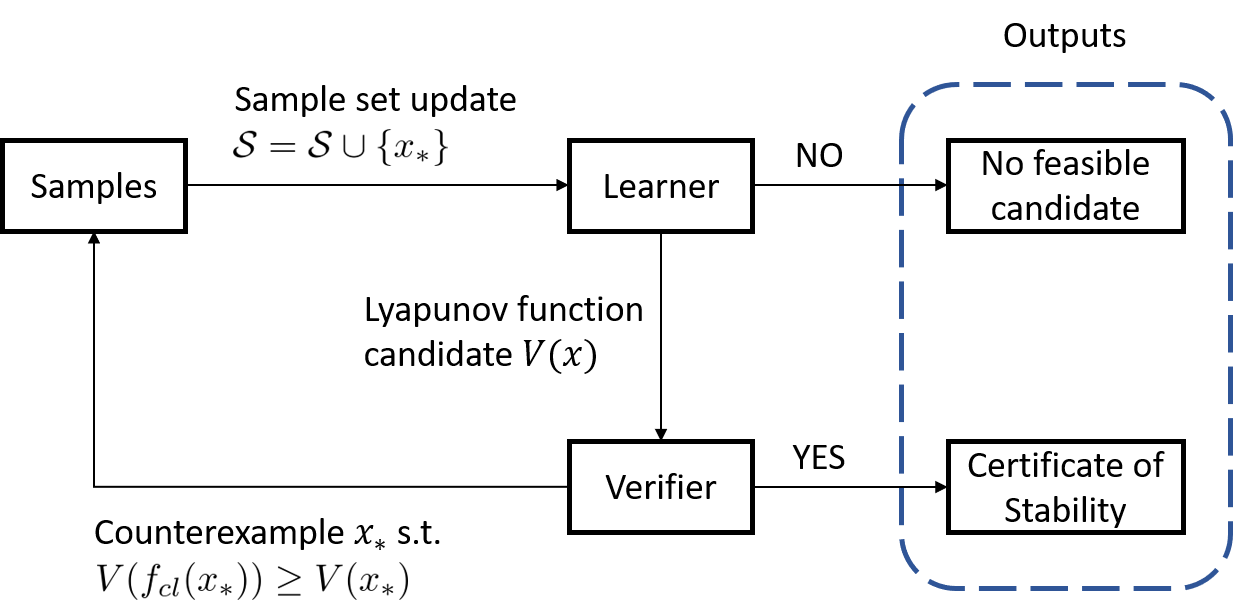}	
	\caption{Learning-based method for Lyapunov function synthesis.}
	\label{fig:learner_verifier}
\end{figure}

\subsection{Quadratic Lyapunov functions}
We begin by considering a quadratic Lyapunov function candidate of the form      
\begin{equation} \label{eq:Lyap_parameterization}
V(x; P) = x^\top P x
\end{equation}
with the matrix parameter $P \in \mathbb{S}^{n_x}_{++}$ and domain $\mathcal{R}$ the same as that of the PWA system~\eqref{eq:hybrid_dyn}. We aim to find a $P$ such that $V(x; P)$ would satisfy the constraint~\eqref{eq:general_Lyap_cond_2}. Since we can always scale $P$ while satisfying constraint~\eqref{eq:general_Lyap_cond_2}, without loss of generality, we describe the set of valid Lyapunov function candidate parameters as
\begin{align} \label{eq:valid_Lyap_class}
\mathcal{F} = \{ P \in \mathbb{S}^{n_x} \vert \ 0 \prec P \preceq I, \ f_{cl}(x)^{\top} P f_{cl}(x) - x^\tp P x < 0, \ \forall x \in \mathcal{X} \setminus \{0 \} \}.
\end{align}

We observe that $\mathcal{F}$ is a convex set defined by semidefinite constraints and infinitely many linear constraints on $P$. Although, theoretically, a feasible point in $\mathcal{F}$ can be obtained by solving a convex semi-infinite feasibility problem~\cite{polak2012optimization}, such an approach is numerically intractable. This motivates sample-based approaches which instead require the constraints in~\eqref{eq:valid_Lyap_class} to only hold for a finite set of samples $\mathcal{S} = \{x^1, x^2, \cdots, x^N\} \subset \mathcal{X}$, resulting in an over-approximation~\footnote{For numerical reasons, we define the over-approximation $\tilde{\mathcal{F}}$ of $\mathcal{F}$ as a compact set. As will be shown in the following sections, our proposed method will find a feasible point in the interior of the closure of $\mathcal{F}$ using barrier functions. Therefore, if $\mathcal{F}$ is non-empty, our method will find a feasible point in $\mathcal{F}$ satisfying all the strict inequalities.} of $\mathcal{F}$:
\begin{equation} \label{eq:valid_Lyap_class sample}
\mathcal{F} \subset \tilde{\mathcal{F}} = \{ P \in \mathbb{S}^{n_x} \ \vert \ 0 \preceq P \preceq I, f_{cl}(x)^{\top} P f_{cl}(x) - x^\tp P x \leq 0, \ \forall x \in \mathcal{S}  \}.
\end{equation}
Finding a feasible point in $\tilde{\mathcal{F}}$ is now a tractable convex feasibility program. However, there is no guarantee that the resulting solution $P$ would correspond to a valid Lyapunov function candidate, even if the number of samples approaches infinity. In this paper, we propose an efficient sampling strategy based on the analytic center cutting-plane method (ACCPM) to iteratively grow the sample set and refine the set $\tilde{\mathcal{F}}$ until we find a feasible point in $\mathcal{F}$. In the next subsections, we describe the proposed approach and provide finite-step termination and validity guarantees.

\subsection{The analytic center cutting-plane method}

Cutting-plane methods~\cite{atkinson1995cutting, elzinga1975central,boyd2007localization} are iterative algorithms to find a point in a target convex set $\mathcal{F}$ or to determine whether $\mathcal{F}$ is empty. In these methods, we have no information of $\mathcal{F}$ except for a ``cutting-plane oracle'' described below. At each iteration $i$, we have a ``localization'' set $\tilde{\mathcal{F}_{i}}$ defined by a finite set of inequalities that over-approximates the target set, i.e., $\mathcal{F} \subseteq \tilde{\mathcal{F}_{i}}$. If $\tilde{\mathcal{F}_{i}}$ is empty, then we have a proof that the target set $\mathcal{F}$ is also empty. Otherwise, we query the oracle at a point $P^{(i)} \in \tilde{\mathcal{F}_{i}}$. If $P^{(i)} \in \mathcal{F}$, the oracle returns `yes' and the algorithm terminates; if $P^{(i)} \notin \mathcal{F}$, it returns `no' together with a separating hyperplane that separates $P^{(i)}$ and $\mathcal{F}$. In the latter case, the cutting-plane method updates the localization set by $\tilde{\mathcal{F}}_{i+1} = \tilde{\mathcal{F}_{i}} \cap \{\text{half space defined by the separating hyperplane}\}$ (see Fig.~\ref{fig:cutting_plane_demo}). This process continues until either a point in the target set is found or the target set is certified to be empty.

Based on how the query point is chosen, different cutting-plane methods have been proposed including the center of gravity method~\cite{Lev65}, the maximum volume ellipsoid (MVE) cutting-plane method~\cite{Tarasov1988}, the Chebyshev center cutting-plane method~\cite{elzinga1975central}, the ellipsoid method~\cite{khachiyan1980polynomial, ellipsoidYudin}, the volumetric center cutting-plane method~\cite{vaidya1989new}, and the analytic center cutting-plane method~\cite{goffin1993computation, nesterov1995cutting, atkinson1995cutting}. In this paper, the query point $P^{(i)}$ is chosen as the analytic center of $\tilde{\mathcal{F}}_i$ according to the ACCPM, since it allows the localization set $\tilde{\mathcal{F}_{i}}$ to be described by linear matrix inequalities. 

\begin{figure}
	\centering
	\includegraphics[width = 0.6 \textwidth]{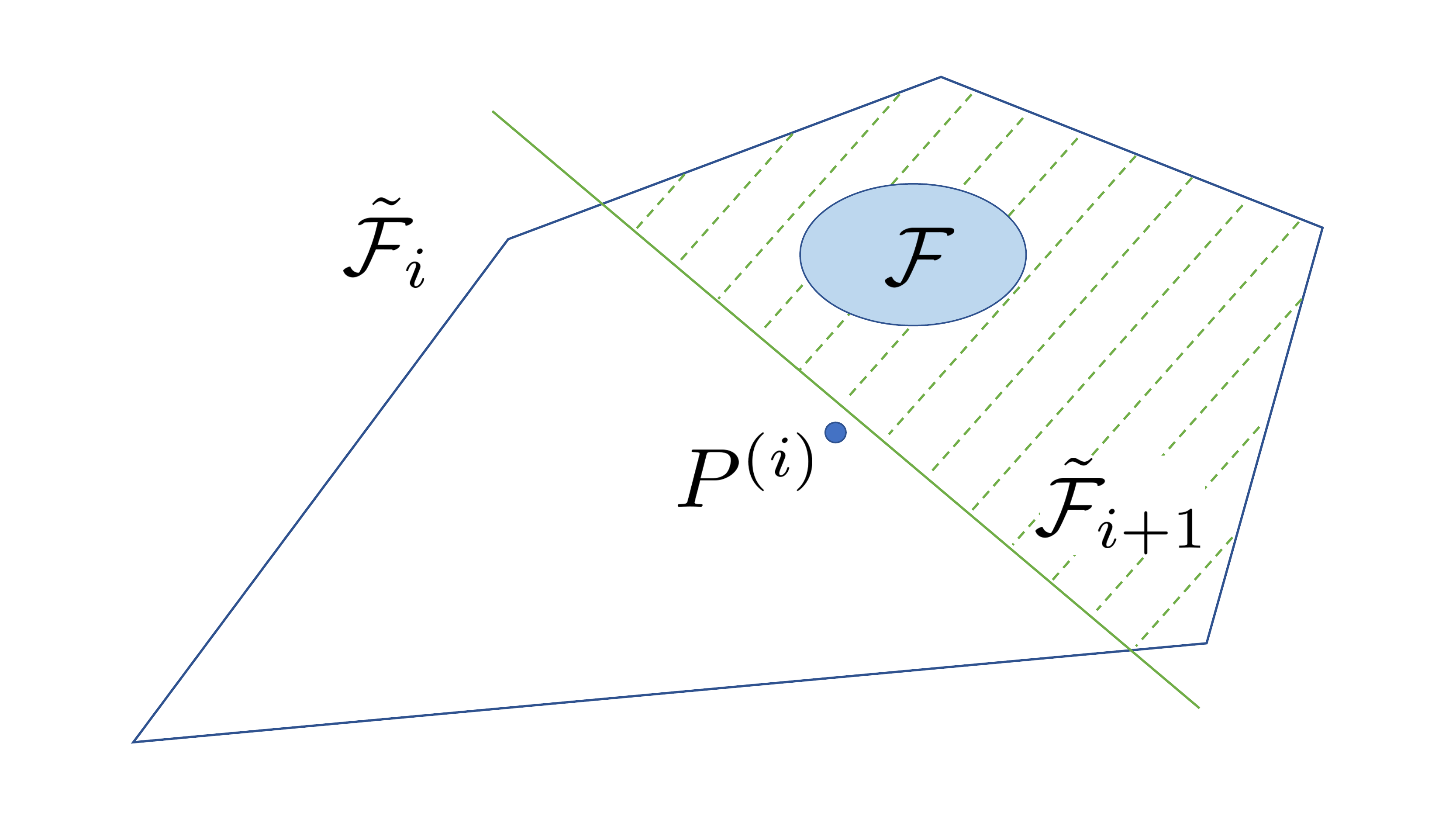}	
	\caption{In the $i$-th iteration, the query point $P^{(i)}$ is chosen as the ``center'' of the localization set $\tilde{\mathcal{F}}_i$ in order to guarantee the removal of a large enough portion of $\tilde{\mathcal{F}}_i$ from the search space whenever a separating hyperplane (green solid line) is given. The localization set is then updated to $\tilde{\mathcal{F}}_{i+1}$ (shaded set) which contains the target set $\mathcal{F}$.}
	\label{fig:cutting_plane_demo}
\end{figure}

\begin{definition}[Analytic center~\cite{boyd2004convex}] \label{def:analytic_center}
	The analytic center $x_{ac}$ of the set
	$\{x \vert f_i(x) \leq 0, i = 1, \cdots, m, \ Fx = g\}$,	is defined as the optimal solution of the convex problem
	\begin{equation} \label{eq:analytic_center}
	\begin{aligned}
	\underset{x}{\mathrm{minimize}} & \quad - \sum_{i=1}^m  \log(-f_i(x))  \quad \emph{subject to} \quad Fx = g.
	\end{aligned}
	\end{equation}
\end{definition}

To implement the ACCPM, we introduce a learner, which proposes Lyapunov function candidates based on a set of samples, and a verifier, which serves as a cutting-plane oracle. We describe each one of these two elements below.

\subsection{The learner}
Let $\mathcal{S} = \{x^1, x^2, \cdots, x^N \} \subset \mathcal{X}$ be a collection of samples from $\mathcal{X}$. The step-wise increment in the Lyapunov function as a function of the current state and the matrix $P$ can be written as
\begin{equation*}
\Delta V(x,P):= V(x_{+};P)-V(x;P) = f_{cl}(x)^\top P f_{cl}(x) - x^\top P x,
\end{equation*}
which is a linear function in $P$. Based on the sample set $\mathcal{S}$, we construct an outer-approximation of $\mathcal{F}$ as 
\begin{equation}\label{eq:F_0}
\tilde{\mathcal{F}} = \{P \in \mathbb{S}^{n_x} \vert \ 0 \preceq P \preceq I, \ \Delta V(x,P) \leq 0, \ \forall x \in \mathcal{S} \}.
\end{equation}
The localization set $\tilde{\mathcal{F}}$ represents the learner's knowledge about $\mathcal{F}$ by observing the one-step state transitions of the dynamical system from the sample set $\mathcal{S}$. Utilizing the ACCPM, the learner proposes a Lyapunov function candidate $V(x; P_{ac})$ with $P_{ac}$ as the analytic center of $\tilde{\mathcal{F}}$:
\begin{equation}\label{eq:ac_optimization}
P_{ac}: = \underset{P}{\text{argmin}} \quad -\sum_{x\in \mathcal{S}} \log( -\Delta V(x,P) ) - \log\det (I - P) - \log \det (P).
\end{equation}
We denote the objective function in~\eqref{eq:ac_optimization} as $\phi(\tilde{\mathcal{F}})$ and call it the \emph{potential function} of the set $\tilde{\mathcal{F}}$. If \eqref{eq:ac_optimization} is infeasible, then we have a proof that no quadratic Lyapunov function  exists for the closed-loop system. Otherwise, the learner proposes $V(x; P_{ac}) = x^\top P_{ac} x$ as a Lyapunov function candidate. Due to the log-barrier function in~\eqref{eq:ac_optimization}, $P_{ac}$ is in the interior of $\tilde{\mathcal{F}}$. At initialization, we have an empty sample set, i.e., $\mathcal{S} = \emptyset$ and the potential function simply becomes $\phi(\tilde{\mathcal{F}}) = -\log \det(I - P) - \log \det (P)$ with $P_{ac} = \frac{1}{2}I$ as the analytic center. 

\subsection{The verifier}
Given a Lyapunov function candidate $V(x; P^{(i)})$ proposed by the learner at iteration $i$, the verifier either ensures that this function satisfies the constraints in~\eqref{eq:general_Lyap_cond_1} and~\eqref{eq:general_Lyap_cond_2}, or returns a state where constraint~\eqref{eq:general_Lyap_cond_1} or~\eqref{eq:general_Lyap_cond_2} is violated as a counterexample. Since the log-barrier function in~\eqref{eq:ac_optimization} guarantees $P^{(i)} \succ 0$, constraint~\eqref{eq:general_Lyap_cond_1} is readily satisfied and the verifier must check the violation of constraint~\eqref{eq:general_Lyap_cond_2}. This can be done by solving the optimization problem
\begin{alignat}{2} \label{eq:LTI_MIQP_0}
&\underset{x \in \mathcal{X} \setminus \{0\}}{\mathrm{maximize}} \quad && \Delta V(x,P^{(i)}).
\end{alignat}
Next, we will show that for PWA systems and ReLU NN controllers, the optimization problem~\eqref{eq:LTI_MIQP_0} can be formulated as a mixed-integer quadratic program (MIQP). This is based on the fact that both the PWA dynamics $\psi(x, u)$ and the ReLU NN controller can be described by a set of mixed-integer linear constraints.

\subsubsection{Mixed-integer formulation of PWA dynamics}
\label{sec:disjunctive_PWA}

To formulate optimization problems involving hybrid dynamics, e.g., the optimal control problem of PWA systems, we need to describe system~\eqref{eq:hybrid_dyn} in a form compatible with optimization solvers. In this paper, we use disjunctive programming~\cite{marcucci2019mixed} to describe the PWA dynamics~\eqref{eq:hybrid_dyn}, which exploits the underlying geometry of the system. Other modeling approaches, e.g., mixed logical dynamical (MLD) formulation~\cite{borrelli2017predictive}, are also applicable and our proposed method can be easily extended to these modeling approaches since all of them encode the PWA dynamics by mixed-integer linear constraints.  

Recall that the PWA dynamics~\eqref{eq:hybrid_dyn} is given by $x_{+} = \psi_i(x, u) = A_i x + B_i u + c_i$ if $x \in \mathcal{R}_i$. Equivalently, we can describe the PWA dynamics in $\mathcal{R}_i$ in the lifted space of $(x, u, x_+) \in \mathbb{R}^{2n_x + n_u}$ as $(x, u, x_+) \in \gr(\psi_i)$, where $\gr(\psi_i)$ is the graph of $\psi_i$ which is defined as the following polytope in $\mathbb{R}^{2n_x + n_u}$:
\begin{equation} \label{eq:graph}
\gr(\psi_i) = \{ (x, u, x_+) \vert P_i(x, u, x_+) \leq q_i \}, \text{ with }
P_i = \begin{bmatrix}
A_i & B_i & -I \\ 
-A_i & -B_i & I \\
F_i & 0 & 0 \\
0 & G_u & 0 
\end{bmatrix}, \quad 
q_i = \begin{bmatrix}
-c_i \\ c_i \\h_i \\ h_u
\end{bmatrix}.
\end{equation}
The first two block rows of $P_i$ and $q_i$ define the dynamics $x_{+} = A_i x + B_i u + c_i$ and the third block row constrains $x$ to lie in $\mathcal{R}_i = \{x \vert F_i x \leq h_i\}$. In the last block row, we add an artificial constraint $u \in \bar{\mathcal{U}} = \{u \vert G_u u \leq h_u\}$ on the control input, where $\bar{\mathcal{U}}$ is a bounded set in $\mathbb{R}^{n_u}$. This guarantees that $\gr(\psi_i)$ have a common recession cone $\gr(\psi_i)_\infty = \{(0, 0, 0) \}$ for all $i \in \mathcal{I}$. Details of choosing $\bar{\mathcal{U}}$ will be discussed in Section~\ref{sec:implement_ACCPM}. Following from the definition of $\gr(\psi_i)$, we have 
\begin{equation*}
(x, u , x_+) \in \gr(\psi_i) \Leftrightarrow x_+ = \psi_i(x,u), \ x\in \mathcal{R}_i.
\end{equation*}
By denoting the PWA dynamics collectively as $x_{+} = \psi(x, u)$, we define the graph of $\psi(x,u)$ as $\gr(\psi) := \bigcup_{i\in \mathcal{I}} \gr (\psi_i)$. Then we have
\begin{equation*}
(x, u, x_+) \in \gr(\psi) \Leftrightarrow x_+ = \psi(x, u),
\end{equation*}
which indicates that the evolution of the PWA dynamics is constrained in a union of disjunctive sets $\gr(\psi)$. As a result, optimization over  $x, u, x_{+}$ subject to the PWA system dynamics is equivalent to optimizing over the set $\gr(\psi)$. Next, we will show that $\gr(\psi)$ lends itself to a mixed-integer formulation.

\begin{definition}[Mixed-integer formulation of a set~\cite{marcucci2019mixed}] \label{def:MIF}
	For a set $\mathcal{Q} \subset \mathbb{R}^{n_z}$, consider the set $\mathcal{L}_\mathcal{Q} \subseteq \mathbb{R}^{n_z + n_\lambda} \times \mathbb{Z}^{n_\mu}$ in a lifted space given by 
	\begin{equation}
	\mathcal{L}_\mathcal{Q} = \{ (z \in \mathbb{R}^{n_z}, \lambda \in \mathbb{R}^{n_\lambda}, \mu \in \mathbb{Z}^{n_\mu})  \vert \ell(z, \lambda, \mu) \leq v\},
	\end{equation}
	with a function $\ell: \mathbb{R}^{n_z + n_\lambda} \times \mathbb{Z}^{n_\mu} \rightarrow \mathbb{R}^{n_\ell}$ and a vector $v \in \mathbb{R}^{n_\ell}$. The set $\mathcal{L}_\mathcal{Q}$ is a mixed-integer formulation of $\mathcal{Q}$ if $\Proj_z(\mathcal{L}_\mathcal{Q}) = \mathcal{Q}$. If the function $\ell$ is linear (convex), we call the related formulation MIL (MIC). 
\end{definition} 
In Definition~\ref{def:MIF}, if $\mathcal{Q}$ has a mixed-integer formulation $\mathcal{L}_\mathcal{Q}$, then $\mathcal{Q}$ can be represented by the inequalities $\ell(z, \lambda, \mu) \leq v$, which means
\begin{equation*}
z \in \mathcal{Q} \Leftrightarrow \exists \lambda, z \text{ s.t. } \ell(z, \lambda, \mu) \leq v.
\end{equation*}
For the PWA dynamics $x_+ = \psi(x,u)$, we construct an MIL formulation of $\gr(\psi)$, known as the convex-hull formulation~\cite{marcucci2019mixed}, in the following lemma.

\begin{lemma} \label{lemma:MIL}
	For PWA system $x_+ = \psi(x, u)$ defined in~\eqref{eq:hybrid_dyn}, let $x_i \in \mathbb{R}^{n_x}, u_i\in \mathbb{R}^{n_u}, \mu_i \in \{0, 1\}$ for $i \in \mathcal{I}$. Then the set $\mathcal{L}_{\gr(\psi)} = \{(x, u, x_+), \{x_i\}, \{u_i\} , \{\mu_i\} \vert \eqref{eq:sharp_graph}\}$ defined by the constraints 
	\begin{equation}\label{eq:sharp_graph}
	\begin{aligned}
	&F_i x_i \leq \mu_i h_i, \ G_u u_i \leq \mu_i h_u, \ \mu_i \in \{0, 1\}, \ \forall i \in \mathcal{I} \\
	&(1, x, u, x_+) = \sum_{i \in \mathcal{I}} (\mu_i, x_i, u_i, A_i x_i + B_i u_i + \mu_i c_i), \\
	\end{aligned}
	\end{equation}
	is an MIL formulation of $\gr(\psi)$.
\end{lemma}
\begin{proof}
	Since $\bar{\mathcal{U}}$ and $\mathcal{R}_i$ are bounded, $\gr(\psi_i)$ share the common recession cone $\gr(\psi_i)_{\infty} = \{ (0, 0, 0) \}$ for all $i\in \mathcal{I}$. By~\cite[Corollary $3.6$]{marcucci2019mixed}, the set $\mathcal{L}_{\gr(\psi)}$ in the lemma is an MIL formulation of $\gr(\psi)$.
\end{proof}
In constraints~\eqref{eq:sharp_graph}, the binary variable $\mu_i$ can be interpreted as an indicator of the current mode of the state $x$. For instance, let $\mu_i = 1$ and consequently $\mu_j = 0, \forall j \neq i$. Since $F_j x_j \leq \mu_j h_j = 0 \Rightarrow x_j = 0, \forall j \neq i$, we have $x = \sum_{k \in \mathcal{I}} x_k = x_i$. Then the state $x$ satisfies $F_i x \leq h_i$, i.e., $x \in \mathcal{R}_i$. Similarly, $G_u u_j \leq \mu_j h_u = 0 \Rightarrow u_j = 0, \forall j \neq i$ and $u = \sum_{k \in \mathcal{I}} u_k = u_i$. It then follows $x_+ = \sum_{k\in\mathcal{I}} A_k x_k + B_k u_k + \mu_k c_k = A_i x + B_i u + c_i = \psi_i(x, u)$. The MIL formulation of $\gr(\psi)$ in~\eqref{eq:sharp_graph} allows us to represent the PWA dynamical constraint $x_+ = \psi(x, u)$ through a finite number of mixed-integer linear constraints.

\subsubsection{Mixed-integer formulation of ReLU NN} \label{subsub: Mixed integer formulation of ReLU NN}

Consider a scalar ReLU function $y=\max(0,x)$ where $\underline{x} \leq x \leq \bar{x}$. Then it can be shown that the ReLU function admits the following mixed-integer representation~\cite{tjeng2017evaluating},
\begin{align}
y=\max(0,x), \ \underline{x} \leq x \leq \bar{x} \iff \{y \geq 0, \ y \geq x, y \leq x - \underline{x} (1-t), \ y \leq \bar{x} t, \ t \in \{0,1\}\},
\end{align}
where the binary variable $t \in \{0,1\}$ is an indicator of the activation function being active ($y=x$) or inactive ($y=0$). Now consider a ReLU network described by the equations in \eqref{eq:nn_controller}. Suppose $\underline{m}^{\ell}$ and $\bar{m}^{\ell}$ are known element-wise lower and upper bounds on the input to the $(\ell+1)$-th activation layer, i.e., $\underline{m}_{\ell} \leq W_{\ell} z_{\ell} +b_{\ell}\leq \bar{m}_{\ell}$. Then the neural network equations are equivalent to a set of mixed-integer constraints:
\begin{align} \label{eq:MIL_NN}
z_{\ell+1} = \max(W_\ell z_\ell + b_\ell,0) \iff \begin{cases}
z_{\ell+1} \geq W_\ell z_\ell + b_\ell \\ 
z_{\ell+1} \leq W_\ell z_\ell + b_\ell - \mathrm{diag}(\underline{m}_\ell) (\mathbf{1}-t_\ell) \\
z_{\ell+1} \geq 0 \\
z_{\ell+1} \leq \mathrm{diag}(\bar{m}_\ell)  t_\ell,
\end{cases} 
\end{align}
where $t_{\ell} \in \{0,1\}^{n_{\ell+1}}$ is a vector of binary variables for the $(\ell+1)$-th activation layer and $\mathbf{1}$ denotes the vector of all ones. 
We note that the element-wise pre-activation bounds $\{\underline{m}_{\ell} , \bar{m}_{\ell}\}$ can be found by, for example, interval bound propagation or linear programming assuming known bounds on the input of the neural network \cite{weng2018towards,zhang2018efficient,hein2017formal,wang2018efficient,wong2018provable}. 
Since state-of-the-art solvers for mixed-integer programming are based on the Branch$\&$Bound algorithm~\cite{wolsey1999integer}, tight upper/lower bounds $\{\bar{m}_\ell, \underline{m}_\ell\}$ will allow the Branch$\&$Bound algorithm to prune branches more efficiently and reduce the total running time.

\subsubsection{Verifier construction through MIQP}
\label{sec:verifier_MIQP}
Having the mixed-integer formulation of the PWA dynamics and the ReLU network at our disposal, we can formulate the verifier's optimization problem~\eqref{eq:LTI_MIQP_0} as the following MIQP:
\begin{subequations} \label{eq:pwa_MIQP}
	\allowdisplaybreaks
	\begin{align}
	\begin{split} \label{eq:MIQP_obj} 
	\underset{ \begin{subarray}{c}
		x, u, y, \{z_\ell \}, \{t_\ell\}  \\
		\{\mu_i\}, \{x_i\}, \{u_i\}
		\end{subarray} }{\text{maximize}} & \quad  y^\top P^{(i)} y  - x^\top P^{(i)} x
	\end{split}\\
	\begin{split}  \label{eq:MIQP_constr_1}
	\text{subject to} 
	& \quad F_\mathcal{X} x \leq h_\mathcal{X}
	\end{split}\\
	\begin{split} \label{eq:MIQP_constr_guard}
	& \quad \lVert x \rVert_\infty \geq \epsilon 
	\end{split} \\
	\begin{split} \label{eq:MIQP_constr_2}
	\quad z_0 = x
	\end{split}\\
	\begin{split}
	& \quad \text{for } \ell = 0, \cdots, L-1:
	\end{split}\\
	\begin{split}
	& \quad \quad z_{\ell+1} \geq 0, \ z_{\ell+1} \geq W_\ell z_{\ell} + b_\ell 
	\end{split}\\
	\begin{split} \label{eq:MIQP_bigM_1}
	& \quad \quad z_{\ell+1} \leq W_\ell z_{\ell} + b_\ell - \mathrm{diag}(\underline{m}_\ell) (\mathbf{1}-t_\ell)
	\end{split}\\
	\begin{split} \label{eq:MIQP_bigM_2}
	& \quad \quad z_{\ell+1} \leq \mathrm{diag}(\bar{m}_\ell) t_\ell, \ t_\ell \in \{0, 1 \}^{n_\ell}
	\end{split}\\
	\begin{split} \label{eq:MIQP_constr_3}
	& \quad u = W_L z_L + b_L
	\end{split}\\
	\begin{split} \label{eq:MIQP_constr_4}
	& \quad F_i x_i \leq \mu_i h_i, \ \forall i \in \mathcal{I} 
	\end{split}\\
	\begin{split}
	& \quad G_u u_i \leq \mu_i h_u, \forall i \in \mathcal{I}
	\end{split}\\
	\begin{split}
	& \quad (1, x, u, y) = \sum_{i \in \mathcal{I}} (\mu_i, x_i, u_i, A_i x_i + B_i u_i + \mu_i c_i)
	\end{split}\\
	\begin{split}\label{eq:MIQP_constr_5}
	& \quad \mu_i \in \{0, 1\}, \ \forall i \in \mathcal{I}.
	\end{split}
	\end{align}
\end{subequations}
Constraint~\eqref{eq:MIQP_constr_1} restricts the search space to $\mathcal{X}$; constraint~\eqref{eq:MIQP_constr_guard} with $\epsilon > 0$ excludes a neighborhood around the origin\footnote{For choice of $\epsilon$ see the discussion in Section~\ref{sec:implement_ACCPM}.} to verify constraint~\eqref{eq:general_Lyap_cond_2} with strict inequalities;  constraints~\eqref{eq:MIQP_constr_2} to~\eqref{eq:MIQP_constr_3} exactly model the ReLU network controller $u = \pi(x)$ with known layer-wise upper and lower bounds $\{\bar{m}_\ell, \underline{m}_\ell\}$ (see Section~\ref{subsub: Mixed integer formulation of ReLU NN}); and constraints~\eqref{eq:MIQP_constr_4} to~\eqref{eq:MIQP_constr_5} describe the PWA dynamics $y = \psi(x, u)$ (see Section~\ref{sec:disjunctive_PWA}). The objective function~\eqref{eq:MIQP_obj} aims to find a counterexample where the Lyapunov condition~\eqref{eq:general_Lyap_cond_2} is violated with the largest margin.

Denote by $x_*$ the optimal solution and $p^*$ the optimal value of problem~\eqref{eq:pwa_MIQP}. If $p^* < 0$, then we have a proof that $V(x; P^{(i)})$ is a Lyapunov function, since $\Delta V(x, P^{(i)}) < 0$ for all $x \in \mathcal{X} \setminus B_\epsilon$ where $B_\epsilon = \{x \vert \lVert x \rVert_\infty < \epsilon\}$ and we terminate the search. As will be shown in Section~\ref{sec:implement_ACCPM}, the choice of $\epsilon$ guarantees $p^* < 0$ is sufficient to prove the asymptotic stability of the origin. If $p^* \geq 0$, then we reject the Lyapunov function candidate $V(x;P^{(i)})$ with $x_*$ as the counterexample since constraint~\eqref{eq:general_Lyap_cond_2} is violated at a non-zero state $x_* \in \mathcal{X}$ with $\Delta V(x_*, P^{(i)}) = p^* \geq 0$. Meanwhile, we obtain a separating hyperplane $\Delta V(x_*, P) = 0$ in $P$ such that $\Delta V(x_*, P) \leq 0$ for all $P \in \mathcal{F}$. Therefore, the MIQP~\eqref{eq:pwa_MIQP} serves as a cutting-plane oracle.

\subsection{Implementation of the ACCPM}
\label{sec:implement_ACCPM}
With a learner that solves the convex program in ~\eqref{eq:ac_optimization} and a verifier that solves the MIQP in ~\eqref{eq:pwa_MIQP}, we summarize the overall procedure in Algorithm~\ref{alg:ACCPM}. In Algorithm~\ref{alg:ACCPM}, we iteratively grow the sample set $\mathcal{S}_i$ through the interaction between the learner and the verifier to guide the search for Lyapunov functions. Several practical issues of implementing Algorithm~\ref{alg:ACCPM} are addressed below.

\begin{algorithm}[htb]
	\caption{Learning-based Lyapunov function synthesis}\label{alg:ACCPM}
	\hspace*{\algorithmicindent} \textbf{Input:} initial sample set $\mathcal{S}_0 \subset \mathcal{X}$\\
	\hspace*{\algorithmicindent} \textbf{Output:} 
	$P_*$, Status
	\begin{algorithmic}[1]
		\Procedure{LearningLyapunov}{}
		\State $i=0$
		\While {True}
		\State Generate an outer-approximation $\tilde{\mathcal{F}}_i$ from sample set $\mathcal{S}_i$ by~\eqref{eq:F_0}.
		\If {$\tilde{\mathcal{F}}_{i} = \emptyset$} 
		\State Return: Status = Infeasible, $P_* = \emptyset$.
		\EndIf	
		\State $P^{(i)} = \text{AnalyticCenter}(\tilde{\mathcal{F}_{i}})$ \Comment{The learner proposes $P^{(i)} = \arg \min \eqref{eq:ac_optimization}$ }
		\State Query the verifier (cutting-plane oracle) at $P^{(i)}$ \Comment{The verifier solves~\eqref{eq:pwa_MIQP} with $P^{(i)}$}
		\If {Verifier returns `yes'} \Comment{if $\min \eqref{eq:pwa_MIQP} < 0$}
		\State Return: Status = Feasible, $P_* = P^{(i)}$.
		\Else  \Comment{if $\min \eqref{eq:pwa_MIQP} \geq 0$}
		\State Extract the counterexample $x_* = \arg \min_x \eqref{eq:pwa_MIQP}$
		\State $\mathcal{S}_{i+1} = \mathcal{S}_i \cup \{x_*\}$
		\EndIf	
		\State $i = i+1$
		\EndWhile
		\EndProcedure
		\Function{AnalyticCenter}{$\tilde{\mathcal{F}}_i$}
		\State $P_{ac} = $ Optimal solution to \eqref{eq:ac_optimization}
		\State \Return $P_{ac}$
		\EndFunction
	\end{algorithmic}
\end{algorithm}

\medskip


\noindent \textbf{Exclusion of the origin}: In the MIQP~\eqref{eq:pwa_MIQP}, we exclude the $\ell_\infty$-norm ball $B_\epsilon = \{ x \vert \lVert x \rVert_\infty < \epsilon \}$ in order to verify asymptotic stability of the origin. The choice of $\epsilon$ is based on the observation that the closed-loop system $x_+ = f_{cl}(x)$ is actually a PWA autonomous system and we can identify the local linear dynamics $x_+ = A_{cl}x$ in the polytopic partition $\mathcal{D}_0$ that contains the origin~\cite{karg2020stability}. It follows from linear system theory that the origin is asymptotically stable if and only if all eigenvalues of $A_{cl}$ have magnitudes less than $1$~\cite{chen1995linear}. As a result, we choose $\epsilon$ small enough such that $B_\epsilon \subseteq \mathcal{D}_0$ and divide the asymptotic stability verification into two steps: use MIQP~\eqref{eq:pwa_MIQP} to show the convergence of closed-loop trajectories to $B_\epsilon$ and use $A_{cl}$ to prove the convergence to the origin. Since the stability of $A_{cl}$ can be easily verified, our goal is to find a Lyapunov function that certifies the convergence of closed-loop trajectories to $B_\epsilon$. 

\medskip

\noindent \textbf{Solvability of nonconvex MIQP}: The optimization problem~\eqref{eq:pwa_MIQP} is a nonconvex MIQP since the quadratic objective function~\eqref{eq:MIQP_obj} is indefinite. Therefore, the relaxation of the problem after removing the integer constraints would result in a nonconvex quadratic program. Nonconvex MIQP can be solved to global optimality through Gurobi v$9.0$~\cite{gurobi} by transforming the nonconvex quadratic expression into a bilinear form and applying spatial branching~\cite{belotti2013mixed}. More information on solving nonconvex mixed-integer nonlinear programming can be found in~\cite{vigerske2013decomposition, tawarmalani2013convexification, belotti2009branching}. In this paper, we rely on Gurobi to solve the nonconvex MIQP~\eqref{eq:pwa_MIQP} automatically.

\medskip

\noindent \textbf{Choice of upper/lower bounds $\{\bar{m}_\ell, \underline{m}_\ell \}$ }: As discussed in Section~\ref{subsub: Mixed integer formulation of ReLU NN}, tight element-wise upper and lowers bounds on the pre-activation values would lead to a more efficient pruning in the Branch$\&$Bound algorithm and consequently reduce the overall running time. Finding tighter bounds, however, would require more computation. This trade-off has been explored in the context of neural network verification in \cite{tjeng2017evaluating}.


\medskip

\noindent \textbf{Choice of artificial control input constraint}: In the definition of $\gr(\psi_i)$, we introduced the artificial control input constraint $\bar{\mathcal{U}} = \{u \vert G_u u \leq h_u\}$. First, $\bar{\mathcal{U}}$ should be a bounded set to guarantee the validity of~\eqref{eq:sharp_graph}. Second, $\bar{\mathcal{U}}$ should be an over-approximation of the output range of $\pi(x)$ over $\mathcal{X}$ to guarantee that the MIQP~\eqref{eq:pwa_MIQP} searches over all $x \in \mathcal{X}$. This can be done by choosing $\bar{\mathcal{U}}$ as a box constraint and applying interval arithmetic~\cite{moore2009introduction} to over-approximate the output range of $\pi(x)$. Other output range analysis tools based on LP~\cite{wong2018provable}, SDP~\cite{fazlyab2019safety, raghunathan2018semidefinite}, or MILP~\cite{tjeng2017evaluating} are also applicable.

\medskip

\noindent \textbf{Normalization of the separating hyperplane}: Let $\mathcal{S}_i = \{ x_i^1, \cdots, x_i^N \}$. When generating the localization set $\tilde{\mathcal{F}}_i$, we apply the linear constraint $\Delta V(x_i^j, P) = f_{cl}(x_i^j)^\top P f_{cl}(x_i^j) - x_i^{j, \top} P x_i^j = \langle f_{cl}(x_i^j) f_{cl}(x_i^j)^\top -x_i^{j} x_i^{j, \top}, P \rangle = \langle D_j, P \rangle \leq 0$ for all $1 \leq j \leq N$  where $D_i = f_{cl}(x_i^j) f_{cl}(x_i^j)^\top -x_i^{j} x_i^{j, \top}$. We normalize the linear constraint $\langle D_i, P \rangle \leq 0 $ by $\langle D_i / \lVert D_i \rVert_F, P \rangle \leq 0$ and use the normalized constraint in the analytic center optimization problem~\eqref{eq:ac_optimization}.


\medskip

Next, we provide finite-step termination guarantees for the proposed algorithm based on the analysis of the ACCPM from convex optimization.

\subsection{Convergence analysis}
\label{sec:ACCPM_termination}
The convergence and complexity of the ACCPM have been studied in~\cite{atkinson1995cutting,nesterov1995cutting, ye1992potential, luo2000polynomial,goffin1996complexity, sun2002analytic} under various assumptions on the localization set, the form of the separating hyperplane, whether multiple cuts are applied, etc. Directly related to Algorithm~\ref{alg:ACCPM} and the search for a quadratic Lyapunov function is~\cite{sun2002analytic}, which analyzes the complexity of the ACCPM with a matrix variable and semidefiniteness constraints. Notably, it provides an upper bound on the number of iterations that Algorithm~\ref{alg:ACCPM} can run before termination. In~\cite{sun2002analytic}, it is assumed that
\begin{itemize}
	\item A1: $\mathcal{F}$ is a convex subset of $\mathbb{S}^{n_x}$.
	\item A2: $\mathcal{F}$ contains a non-degenerate ball of radius $\epsilon > 0$, i.e., there exists $P_{center}\in \mathbb{S}^{n_x}$ such that $\{P \in \mathbb{S}^{n_x} \vert \lVert P - P_{center} \rVert_F \leq \epsilon \} \subset \mathcal{F}$ where $\lVert \cdot \rVert_F$ is the Frobenius norm.
	\item A3: $\mathcal{F} \subset \{P \in \mathbb{S}^{n_x} \vert 0 \preceq P \preceq I \}$.
\end{itemize}
Let the ACCPM start with the localization set $\mathcal{F}_0 = \{ P \vert 0 \preceq P \preceq I\}$ (or empty sample set $\mathcal{S}_0$ in~\eqref{eq:ac_optimization}) and initialize the first query point $P^{(0)} = \frac{1}{2}I$ correspondingly. If at iteration $i$, a query point $P^{(i)}$ is rejected by the oracle, a separating hyperplane of the form $\langle D_i, P - P^{(i)} \rangle = 0$ is given and $D_i$ is normalized to satisfy $\lVert D_i \rVert_F = 1$. Then by induction, each localization set $\tilde{\mathcal{F}_{i}}$ with $i \geq 1$ is given by 
\begin{equation*}
\tilde{\mathcal{F}_{i}} = \{ P \vert 0 \preceq P \preceq I, \langle D_j, P \rangle \leq c_j, j = 0, \cdots, i-1 \}.
\end{equation*}
with $D_j$ defining the separating hyperplane at iteration $j$ and $c_j = \langle D_j, P^{(j)} \rangle$. The analytic center of $\tilde{\mathcal{F}_{i}}$ is found by minimizing its potential function
\begin{equation} \label{eq:potential_fcn}
\phi(\tilde{\mathcal{F}_{i}}) = -\sum_{j=0}^{i-1} \log(c_j - \langle D_j, P\rangle) - \log \det(P) - \log \det (I -P),
\end{equation}
i.e., the query point $P^{(i)} =  \arg\min \phi(\tilde{\mathcal{F}_{i}})$ at iteration $i$. Then the ACCPM following the above description achieves the following finite-step termination guarantee.
\begin{theorem}
	\label{thm:ACCPM_termination}
	The analytic center cutting-plane method under assumptions A1 to A3 terminates in at most $O({n_x}^3/\epsilon^2)$ calls to the oracle.
\end{theorem} 
\begin{proof}
	Sketch of proof: For the sequence of localization sets $\tilde{\mathcal{F}}_k$, \cite{sun2002analytic} computes an upper bound on the potential function $\phi(\tilde{\mathcal{F}}_k)$ which is approximately $k \log( \frac{1}{\epsilon})$, and a lower bound on $\phi(\tilde{\mathcal{F}}_k)$ which is proportional to $\frac{k}{2}\log(\frac{k}{n_x^3})$. Since the ACCPM must terminate before the lower bound exceeds the upper bound, we obtain the finite-step termination guarantee in Theorem~\ref{thm:ACCPM_termination}.
\end{proof}

The application of Theorem~\ref{thm:ACCPM_termination} to Algorithm~\ref{alg:ACCPM} is direct. In Algorithm~\ref{alg:ACCPM}, we choose $D_i = f_{cl}(x_*^{(i)}) f_{cl}(x_*^{(i)})^\top - x_*^{(i)} x_*^{(i), \top}$ where $x_*^{(i)}$ denotes the counterexample found at iteration $i$, and normalize it to be of unit norm. In order to be consistent with the ACCPM described in~\cite{sun2002analytic}, we relax the separating hyperplane from $\langle D_i , P \rangle \leq 0$ to $\langle D_i, P \rangle \leq \langle D_i, P^{(i)} \rangle$ (note $\langle D_i, P^{(i)} \rangle \geq 0$ by definition of $D_i$) with which we can directly apply Theorem~\ref{thm:ACCPM_termination} to guarantee that Algorithm~\ref{alg:ACCPM} terminates in at most $O({n_x}^3/\epsilon^2)$ iterations if the set of Lyapunov functions is full-dimensional in the parameter space. However, when the target set $\mathcal{F}$ is empty, we do not have such termination guarantee. In this case, the detection of non-existence of Lyapunov functions depends on finding an empty over-approximation $\tilde{\mathcal{F}}$.

\section{Extensions and variations}
\label{sec:extensions}
\subsection{Piecewise quadratic Lyapunov functions}
\label{sec:PWQLyap}
In the previous section, we used quadratic Lyapunov functions to verify the stability of the closed-loop system. When no globally quadratic Lyapunov function can be found or the results are too conservative, a natural extension is to consider \emph{piecewise quadratic} Lyapunov functions. In this section, we consider this class. Specifically, consider the following piecewise quadratic Lyapunov function candidate
\begin{equation}\label{eq:PWQ_Lyap}
V(x;P) = \begin{bmatrix}
x \\ f_{cl}(x)
\end{bmatrix}^\top P 
\begin{bmatrix}
x \\ f_{cl}(x)
\end{bmatrix},
\end{equation}
with $P \in \mathbb{S}^{2n_x}$. This function is a composition of a quadratic function and the PWA closed-loop dynamics $f_{cl}(x)$ and, hence, is a PWQ function. The parameterization in~\eqref{eq:PWQ_Lyap} is inspired by the non-monotonic Lyapunov function~\cite{ahmadi2008non} and finite-step Lyapunov function methods~\cite{aeyels1998new, bobiti2016sampling}, which relax the monotonically decreasing condition for Lyapunov functions by incorporating states several steps ahead in the construction of Lyapunov function candidates.

By the formulation in~\eqref{eq:PWQ_Lyap}, we parameterize a class of continuous PWQ functions with as many modes as those of the closed-loop system~\eqref{eq:cl_dynamics}, but with only a small-size parameter $P \in \mathbb{S}^{2n_x}$ in order to keep the complexity of running the ACCPM at a manageable level. Similar to Section~\ref{sec:ACCPM_Lyap}, let $\mathcal{F}$ be the set of parameter $P$ such that $V(x;P)$ satisfies the conditions~\eqref{eq:general_Lyap_cond_1} and~\eqref{eq:general_Lyap_cond_2}:
\begin{equation*}
%
\mathcal{F} = \{ P \in \mathbb{S}^{2 n_x} \vert \ 0 \prec P \preceq I, V(f_{cl}(x);P)-V(x;P) < 0, \ \forall x \in \mathcal{X} \setminus \{0  \} \}.
\end{equation*}
Since $\mathcal{F}$ is a convex set in $P$, we can apply the ACCPM to find a feasible point in $\mathcal{F}$.
\subsubsection{Design of the learner}
Consider the sample set with $N$ samples $\mathcal{S} = \{x^{1}, \cdots, x^N\}$. The Lyapunov difference for the closed-loop system is
\begin{equation}
\begin{aligned}
\Delta V(x,P) &= V(f_{cl}(x);P) - V(x;P) \\
& = \begin{bmatrix}
f_{cl}(x) \\ f_{cl}^{(2)}(x)
\end{bmatrix}^\top P 
\begin{bmatrix}
f_{cl}(x) \\ f_{cl}^{(2)}(x)
\end{bmatrix} - 
\begin{bmatrix}
x \\ f_{cl}(x)
\end{bmatrix}^\top P 
\begin{bmatrix}
x \\ f_{cl}(x)
\end{bmatrix}
\end{aligned} 
\end{equation}
where $f_{cl}^{(2)}(x) = f_{cl}(f_{cl}(x))$. The localization set $\tilde{\mathcal{F}}$ is thus given by
\begin{equation}
\tilde{\mathcal{F}} = \{ P \ \vert \ 0 \preceq P \preceq I,  \Delta V(x, P) \leq 0, \forall x \in \mathcal{S}  \},
\end{equation}
which contains $\mathcal{F}$. When $\mathcal{S} = \emptyset$, we have $\tilde{\mathcal{F}} = \{ P \ \vert \ 0 \preceq P \preceq I \}$. The learner finds the analytic center $P_{ac}$ of the localization set $\tilde{\mathcal{F}}$ by solving
\begin{equation} \label{eq:ac_pwq}
\begin{aligned}
\underset{P}{\text{minimize}}  \quad -\sum_{x \in \mathcal{S}} \log  (- \Delta V(x,P)) - \log \det (I - P) - \log \det (P) 
\end{aligned}
\end{equation}

If problem~\eqref{eq:ac_pwq} is infeasible, there is no valid Lyapunov function in the parameterized function class~\eqref{eq:PWQ_Lyap}; otherwise, the solution $P_{ac} =\arg \min \eqref{eq:ac_pwq}$ constructs the Lyapunov function candidate $V(x; P_{ac})$ proposed by the learner. Similarly, we can relax the separating hyperplane as shown in~\eqref{eq:potential_fcn} to apply the finite-step termination guarantee.

\subsubsection{Design of the verifier}
With the PWQ Lyapunov function candidate $V(x; P^{(i)})$ proposed by the learner at iteration $i$, the verifier solves the following MIQP:

\begin{subequations} 
	\allowdisplaybreaks
	\label{eq:pwq_oracle}
	\begin{align}
	\begin{split} \label{eq:pwq_oracle_obj}
	\underset{ \begin{subarray}{c}
		\{x^j\},\{u^j\}, \{\mu^j\}, \{z_\ell^j \}, \{t_\ell^j \}
		\end{subarray} }{\text{maximize}} & \quad 	\begin{bmatrix}
			x^1 \\ x^2
		\end{bmatrix}^\top P^{(i)}
		\begin{bmatrix}
		x^1 \\ x^2
		\end{bmatrix} -
		\begin{bmatrix}
		x^0 \\ x^1
		\end{bmatrix}^\top P^{(i)} 
		\begin{bmatrix}
		x^0 \\ x^1
		\end{bmatrix}
	\end{split}\\
	\begin{split}  \label{eq:pwq_oracle_cond_1}
	\text{subject to} 
	& \quad  F_\mathcal{X} x^0 \leq h_\mathcal{X}
	\end{split}\\
	\begin{split} \label{eq:pwq_oracle_cond_guard}
	&\quad \lVert x^0 \rVert_\infty \geq \epsilon
	\end{split}\\
	\begin{split} \nonumber
	& \quad \text{for } j = 0, 1: 
	\end{split}\\
	\begin{split} \label{eq:pwq_oracle_cond_4}
	& \quad \quad \quad z^j_0 = x^j
	\end{split}\\
	\begin{split}
	& \quad \quad \quad \text{for } \ell = 0, \cdots, L-1:
	\end{split}\\
	\begin{split}
	& \quad \quad \quad \quad \quad z_\ell^j \geq 0, \ z_\ell^j \geq W_\ell z_{\ell -1}^j + b_\ell 
	\end{split}\\
	\begin{split}
	& \quad \quad \quad \quad \quad z_\ell^j \leq W_\ell z_{\ell - 1}^j + b_\ell - \text{diag}(\underline{m}_\ell^j) (1 - t_\ell^j) 
	\end{split}\\
	\begin{split} 
	& \quad \quad \quad \quad \quad z_\ell^j \leq \bar{m}_\ell^j t_\ell^j, \ t_\ell^j \in \{0, 1 \}^{n_\ell} 
	\end{split}\\
	\begin{split} \label{eq:pwq_oracle_cond_5}
	& \quad \quad \quad u^j = W_L z_L^j + b_L
	\end{split}\\
	\begin{split} \label{eq:pwq_oracle_cond_2}
	& \quad \quad \quad F_i x_i^j \leq \mu_i^j h_i, G_u u_i^j \leq \mu_i^j h_u, \ \forall i \in \mathcal{I} 
	\end{split}\\
	\begin{split} 
	& \quad \quad \quad (1, x^j, u^j, x^{j+1}) = \sum_{i \in \mathcal{I}} (\mu_i^j, x_i^j, u_i^j, A_i x_i^j + B_i u_i^j + \mu_i^j c_i)
	\end{split}\\
	\begin{split} \label{eq:pwq_oracle_cond_3}
	& \quad \quad \quad \mu_i^j \in \{0, 1\}, \forall i \in \mathcal{I}
	\end{split}
	\end{align}
\end{subequations}
where we evolve the closed-loop system dynamics for two steps. For each state $x^0 \in \mathcal{X} \setminus B_\epsilon$ (constraint~\eqref{eq:pwq_oracle_cond_1} and~\eqref{eq:pwq_oracle_cond_guard}), constraints~\eqref{eq:pwq_oracle_cond_2} to~\eqref{eq:pwq_oracle_cond_3} model the PWA dynamics $x^1 = \psi(x^0, u^0), x^2 = \psi(x^1, u^1)$ while constraints~\eqref{eq:pwq_oracle_cond_4} to~\eqref{eq:pwq_oracle_cond_5} describe the neural network controller $u^0 = \pi(x^0), u^1 = \pi(x^1)$. The interpretation of the solution of problem~\eqref{eq:pwq_oracle} is the same as that in Section~\ref{sec:verifier_MIQP}.


It is straight forward to parameterize more complex PWQ Lyapunov function candidate classes by concatenating more `future' states $f_{cl}^{(i)}(x), i = 1, 2, \cdots, k$ in~\eqref{eq:PWQ_Lyap}. This only requires modification of the verifier~\eqref{eq:pwq_oracle} by cloning the MIL constraints for $k+1$ times and choosing the ROI $\mathcal{X}$ accordingly to guarantee that $f_{cl}^{(i)}(x)$ is well-defined.


\subsection{Neural network controllers with projection}
\label{sec:constrained_nn}

In Section~\ref{sec:ACCPM_Lyap}, we studied the closed-loop stability of the PWA system~\eqref{eq:hybrid_dyn} with a neural network controller $u =\pi(x)$. In practical applications, we often encounter hard constraints on the state and on the control input (e.g., actuator saturation). In this section, we consider neural networks with a projection layer in feedback interconnection with the LTI system~\eqref{eq:LTI_dyn} with state-space matrices $(A, B)$. Specifically, we assume the projected neural network satisfies the following constraints:
\begin{itemize}
	\item \emph{Control input constraint:} the control input $u$ has to satisfy the polytopic constraint
	\begin{equation}\label{eq:control_constraint}
	u \in \mathcal{U} = \{ u \in \mathbb{R}^{n_u} \vert C_u u \leq d_u\}.
	\end{equation}
	\item \emph{Positive invariance constraint:} the region of interest $\mathcal{X} = \{ x \in \mathbb{R}^{n_x} \vert F_\mathcal{X} x \leq h_\mathcal{X}\} \subseteq \mathcal{R}$ is positive invariant for the closed-loop system.
\end{itemize}
While the control input constraint can always be satisfied by projecting $\pi(x)$ onto $\mathcal{U}$, to make the ROI $\mathcal{X}$ positive invariant we require $\mathcal{X}$ to be a \emph{control invariant set} under the control input constraint $\mathcal{U}$.
\begin{definition}[Control invariant set]
	A set $\mathcal{C} \subseteq \mathcal{R}$ is said to be a control invariant set for the LTI system $x_+ = Ax + Bu$ subject to the control input constraint~\eqref{eq:control_constraint} if
	\begin{equation*}
	\exists u \in \mathcal{U} \text{ such that } x_+ = Ax+Bu \in \mathcal{C}, \quad \forall x \in \mathcal{C}.
	\end{equation*}
\end{definition}
\begin{assumption} \label{assump:control_invariant_set}
	The region of interest $\mathcal{X}$ is a control invariant set for the LTI system~\eqref{eq:LTI_dyn} with control input constraints~\eqref{eq:control_constraint}.
\end{assumption}

\begin{remark}
	While identifying a positively invariant set for the nonlinear closed-loop autonomous system~\eqref{eq:cl_dynamics} is challenging, finding a control invariant set for the LTI system~\eqref{eq:LTI_dyn} with a polytopic control input constraint $\mathcal{U}$ can be done through an iterative algorithm~\cite{borrelli2017predictive}.
\end{remark}
To satisfy the positive invariance constraint, we project the neural network output $u = \pi(x), \ x \in \mathcal{X}$ onto the state-dependent polyhedron
\begin{equation} \label{eq:project_polytope}
\Omega(x) = \{ u \in \mathbb{R}^{n_u} \vert Ax + Bu \in \mathcal{X}, u \in \mathcal{U} \}.
\end{equation}
As a result, the projected control input $u^p = \pi_{proj}(x) := \Proj_{\Omega(x)}(\pi(x))$ is given by the optimal solution to the following convex quadratic program,
\begin{equation} \label{eq:projection}
\begin{aligned}
u^p = \underset{u}{\arg \min} & \quad \frac{1}{2} \lVert u - \pi(x) \rVert_2^2 \\
\text{subject to } & \quad F_\mathcal{X} B u \leq h_\mathcal{X} - F_\mathcal{X} A x\\
& \quad C_u u \leq d_u.
\end{aligned}
\end{equation}
From the definition of $\Omega(x)$, the projected NN controller $\pi_{proj}(x)$ ensures that $\pi_{proj}(x) \in \mathcal{U}$ and $x_+ = Ax + B\pi_{proj}(x) \in \mathcal{X}$ for all $x \in \mathcal{X}$. Hence, the ROI $\mathcal{X}$ is rendered positive invariant under $\pi_{proj}(x)$. The application of the projected NN controller to guarantee constraint satisfaction or set invariance can be found in~\cite{chen2018approximating}. However, the authors do not verify the closed-loop stability under the projected NN controller. In the following, we modify the ACCPM to account for control input projection.

\subsubsection{ACCPM for the projected controller}
Since $\pi_{proj}(x)$ guarantees $\mathcal{X}$ is positive invariant for the closed-loop system, we can verify $\mathcal{X} \subseteq \mathcal{O}$ by searching over quadratic Lyapunov functions through Algorithm~\ref{alg:ACCPM}. For the learner, the analytic center optimization problem~\eqref{eq:ac_optimization} remains unchanged except that the new closed-loop dynamics are applied. For the verifier, we need to construct a new MIQP that reflects the projection operation in~\eqref{eq:projection}.

We observe that the projected control input $u^p$ is the optimal solution to the convex quadratic program~\eqref{eq:projection}. By the KKT conditions, $u^p$ is the solution to the following system of inequalities and equalities:
\begin{subequations}\label{eq:KKT}
	\begin{align}
	\begin{split} \label{eq:KKT_1}
	u^p - \pi(x) + (F_\mathcal{X}B)^\tp \lambda + C_u^\tp \nu = 0
	\end{split} \\
	\begin{split}
	F_\mathcal{X} B u^p + F_\mathcal{X} A x \leq h_\mathcal{X}
	\end{split} \\
	\begin{split}
	C_u u^p \leq d_u 
	\end{split} \\
	\begin{split}\label{eq:KKT_2}
	\lambda \succeq 0, \nu \succeq 0
	\end{split} \\
	\begin{split}\label{eq:KKT_3}
	\lambda^\tp(F_\mathcal{X}B u^p + F_\mathcal{X}Ax - h_\mathcal{X}) = 0
	\end{split} \\
	\begin{split}\label{eq:KKT_4}
	\nu^\tp (C_u u^p - d_u) = 0.
	\end{split}
	\end{align}
\end{subequations}
%
%
Let $n_f$ and $n_c$ be the number of rows of $F_\mathcal{X}$ and $C_u$, respectively. Then $\lambda \in \mathbb{R}^{n_{f}}$, $\nu \in \mathbb{R}^{n_{c}}$ are the Lagrangian dual variables. Constraints~\eqref{eq:KKT_3} and~\eqref{eq:KKT_4} are bilinear but they can be modeled as mixed-integer linear constraints through the big-$M$ method~\cite{simon2016stability}:
\begin{equation} \label{eq:binary_KKT}
\begin{aligned}
& 0 \leq \lambda_i \leq M t_{\lambda, i}, \quad t_{\lambda, i} \in \{0, 1\} \\
& 0 \leq (h_\mathcal{X} - F_\mathcal{X} B u^p - F_\mathcal{X} A x)_i \leq M (1- t_{\lambda, i}), \quad i = 1, \cdots, n_f \\
& 0 \leq \nu_j \leq M t_{\nu, j}, \quad t_{\nu, j} \in \{0, 1\} \\
& 0 \leq (d_u - C_u u^p)_j \leq M (1 - t_{\nu, j}), \quad j = 1, \cdots, n_c
\end{aligned}
\end{equation}
where the subscript $i$ denotes the $i$-th entry of a vector. In~\eqref{eq:binary_KKT}, a single $M$ is chosen large enough for all relevant constraints for simplicity of exposition; however, tighter element-wise lower and upper bounds can be obtained using linear programming. Then for the Lyapunov function candidate $V(x; P^{(i)})$, the verifier for the projected NN controller is given by the following MIQP:
\begin{equation} \label{eq:MIP_proj}
\begin{aligned}
\underset{x, y, u, u^p, \{z_\ell \}, \{ t_\ell \}, t_\lambda, t_\nu }{\text{maximize}} & \quad y^\top P^{(i)} y - x^\top P^{(i)} x \\
\text{subject to } & \quad \eqref{eq:pwa_MIQP}, \quad \eqref{eq:KKT_1} - \eqref{eq:KKT_2}, \quad \eqref{eq:binary_KKT}
\end{aligned}
\end{equation}
where the variable $\pi(x)$ is replaced by $u$ in~\eqref{eq:KKT_1}. The interpretation of~\eqref{eq:MIP_proj} is exactly the same as that of~\eqref{eq:pwa_MIQP}. With the MIQP verifier~\eqref{eq:MIP_proj}, Algorithm~\ref{alg:ACCPM} achieves the termination guarantee described in Section~\ref{sec:ACCPM_termination}. If a valid Lyapunov function is found by the ACCPM, then we verify that $\mathcal{X} \subseteq \mathcal{O}$.

\subsubsection{Projected NN controllers for PWA systems}
The extension of the ACCPM to the analysis of projected neural network controller on a PWA system, however, is not tractable. Although in real world applications, we can project $\pi(x)$ onto $\Omega(x) = \{u \in \mathbb{R}^{n_u} \vert A_{i(x)} x + B_{i(x)} u \in \mathcal{X}, u \in \mathcal{U}\}$ where $i(x)$ is a function mapping $x$ to the mode of the PWA system, we cannot analyze the closed-loop stability with such projected NN controller through ACCPM. The reason is that $u^p = \Proj(\pi(x))$ now is implicitly given by the solution of an MIQP instead of a convex QP~\eqref{eq:projection} in the LTI system example. This prevents us from finding a mixed-integer formulation of $u^p = \Proj(\pi(x))$ through the KKT conditions. 

If only the control input constraint~\eqref{eq:control_constraint} is considered, which means $\pi(x)$ is projected onto $\mathcal{U} = \{u \vert C_u u \leq d_u\}$, the same procedure described in this section can be applied to construct a verifier with the MIL formulation~\eqref{eq:sharp_graph} describing the PWA dynamics and the KKT conditions describing the map $x \mapsto u^p$, and the closed-loop stability can be analyzed through Algorithm~\ref{alg:ACCPM}.



\section{Numerical examples}
\label{sec:numerical}

We demonstrate the application of the ACCPM through two numerical examples: one is a double integrator LTI system and the other is a PWA system modeling an inverted pendulum in contact with an elastic wall. Following from approximate model predictive control (MPC)~\cite{chen2018approximating, karg2020efficient}, we first synthesize an MPC controller $u = \pi_{MPC}(x)$ for the underlying system, and then train a ReLU NN $\pi(x)$ through supervised learning to approximate the MPC controller, i.e., $\pi(x) \approx \pi_{MPC}(x)$. The closed-loop stability of the NN-controlled system is analyzed by synthesizing a Lyapunov function through Algorithm~\ref{alg:ACCPM}. All the simulations are implemented in Python $3.7$ with Gurobi v9.0~\cite{gurobi} on an Intel i7-6700K CPU.

\subsection{A double integrator example}
\label{sec:DI_example}
We first consider model predictive control of the LTI system $x_+ = Ax + Bu, x \in \mathbb{R}^2$ with state and control input constraints $X$ and $U$:
\begin{equation} \label{eq:DI_dynamics}
A = \begin{bmatrix}
1.1 & 1.1 \\ 0 & 1.1 
\end{bmatrix}, \quad B = \begin{bmatrix}
1 \\ 0.5
\end{bmatrix}, \quad X = \{ x \vert \begin{bmatrix}
-5  \\ -5
\end{bmatrix} \leq x \leq \begin{bmatrix}
5 \\ 5
\end{bmatrix} \}, \quad U = \{u \vert  -1 \leq u \leq 1 \}.
\end{equation}
The MPC horizon is given by $T = 20$ together with the stage cost $q(x, u) = x^\top Q x + u^\top R u$, $Q = \text{diag}(1, 1)$, $R = 1$ and terminal cost $p(x) = x^\top P_\infty x$, $P_\infty = \text{DARE}(A, B, Q, R)$, i.e., $P_\infty$ is the solution to the discrete algebraic Riccati equation defined by $(A, B, Q, R)$. The terminal set $X_T$ is chosen as the maximum positive invariant set~\cite[Chapter 10]{borrelli2017predictive} of the closed-loop system $x_+ = (A + B K_\infty) x$ where $K_\infty = -(B^\top P_\infty B + R)^{-1}B^\top P_\infty A$. The choice of $Q, R, P_\infty, X_T$ guarantees the asymptotic stability of the closed-loop system $x_+ = Ax + B\pi_{mpc}(x)$ with a polytopic ROA shown in Fig.~\ref{fig:MPC_ROA} \cite[Chapter 12]{borrelli2017predictive}.

\subsubsection{Synthesis of a NN controller}

By the explicit MPC approach~\cite{bemporad2002explicit, MPT3}, the MPC controller, denoted by $\pi_{mpc}(x)$, is a PWA function with polyhedral partitions as shown in Fig.~\ref{fig:MPC_controller}. By construction of the MPC controller, the ROA of the closed-loop system is a polytopic set shown in Fig.~\ref{fig:MPC_ROA}. We denote the ROA as $\mathcal{X}_0 := \{x \vert F_{\mathcal{X}_0}x \leq h_{\mathcal{X}_0}\}$ and use it as a reference region of interest. 

To obtain a neural network approximate of $\pi_{mpc}(x)$, we sample $650$ states uniformly from $\mathcal{X}_0$ and compute $\pi_{mpc}(x)$ at these samples as the training data. Through supervised learning with Keras~\cite{chollet2015keras}, we obtain a ReLU NN controller $\pi(x)$ with $3$ hidden layers and $10$ neurons in each layer to approximate $\pi_{mpc}(x)$. The bias term in the output layer of the NN is modified such that $\pi(0)=0$. The ReLU NN is visualized in Fig.~\ref{fig:NN_controller}. Next, we want to verify the stability of double integrator system under the NN controller $\pi(x)$.

\begin{figure}[htb!]
	\centering
	\begin{subfigure}{0.32 \textwidth}
		\includegraphics[width = \linewidth]{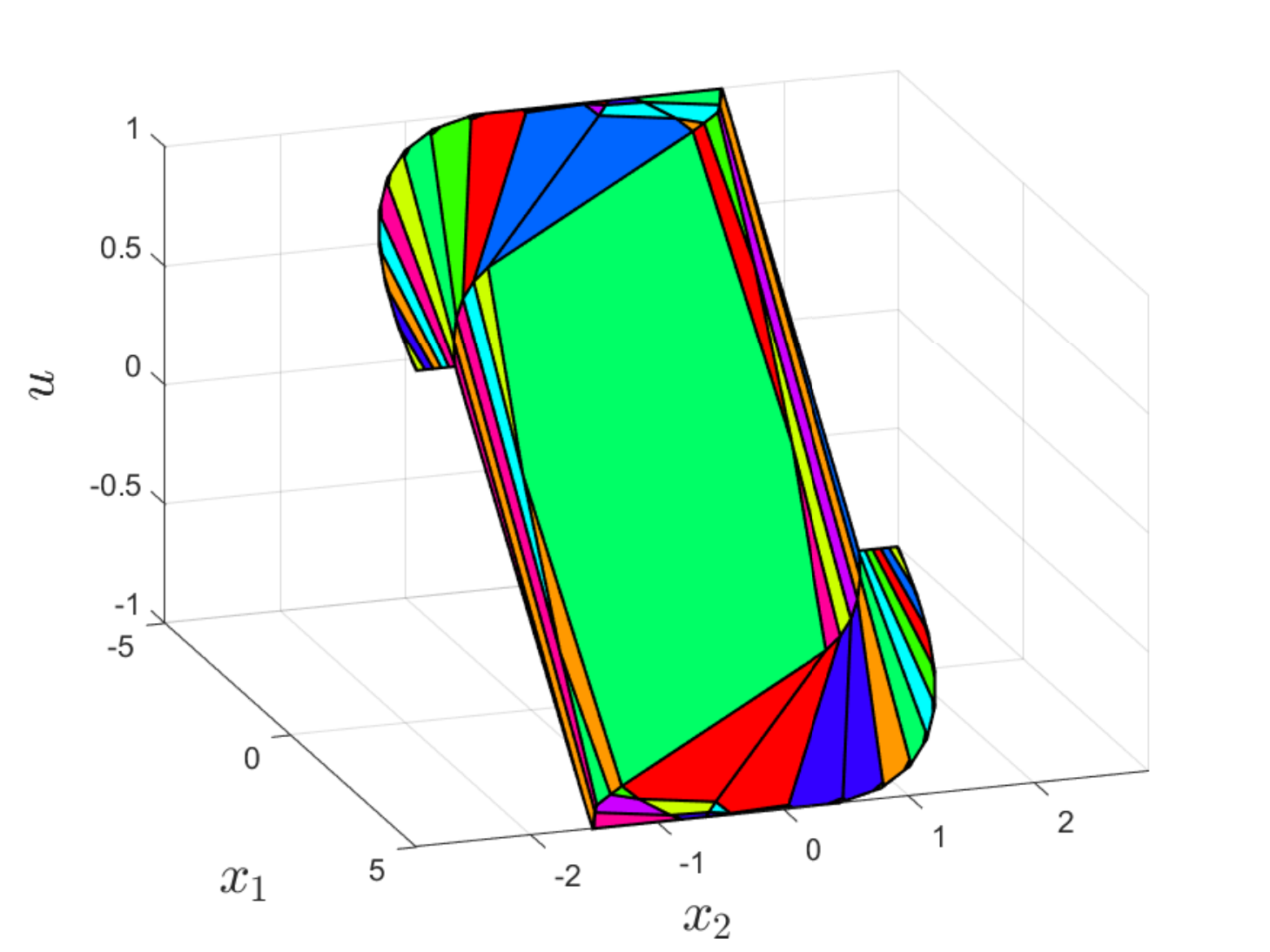}
		\caption{MPC controller}
		\label{fig:MPC_controller}
	\end{subfigure}
	\begin{subfigure}{0.32 \textwidth}
		\includegraphics[width =  \linewidth]{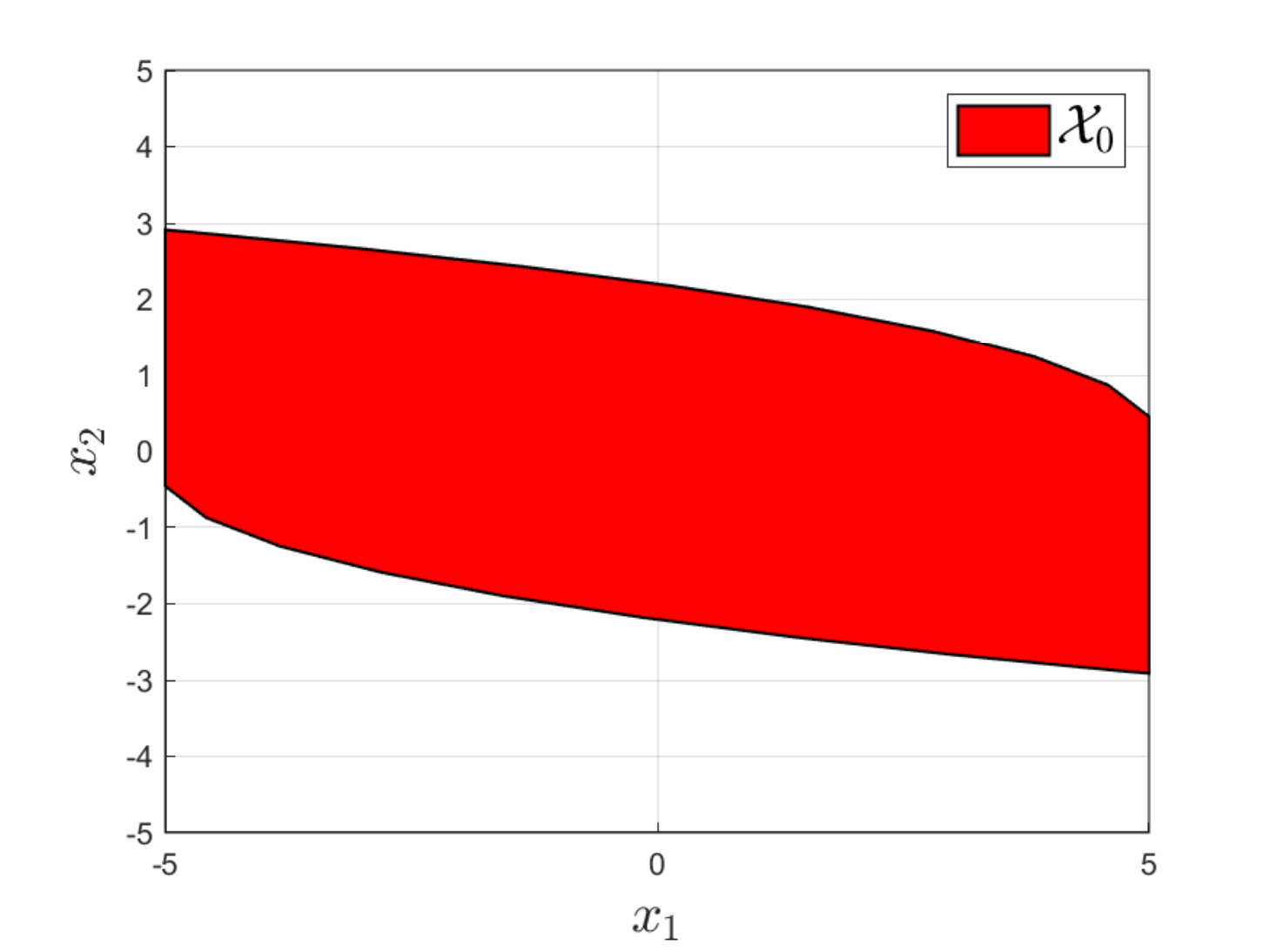}
		\caption{ROA of the MPC controller}
		\label{fig:MPC_ROA}
	\end{subfigure}
	\begin{subfigure}{0.32 \textwidth}
		\includegraphics[width =  \linewidth]{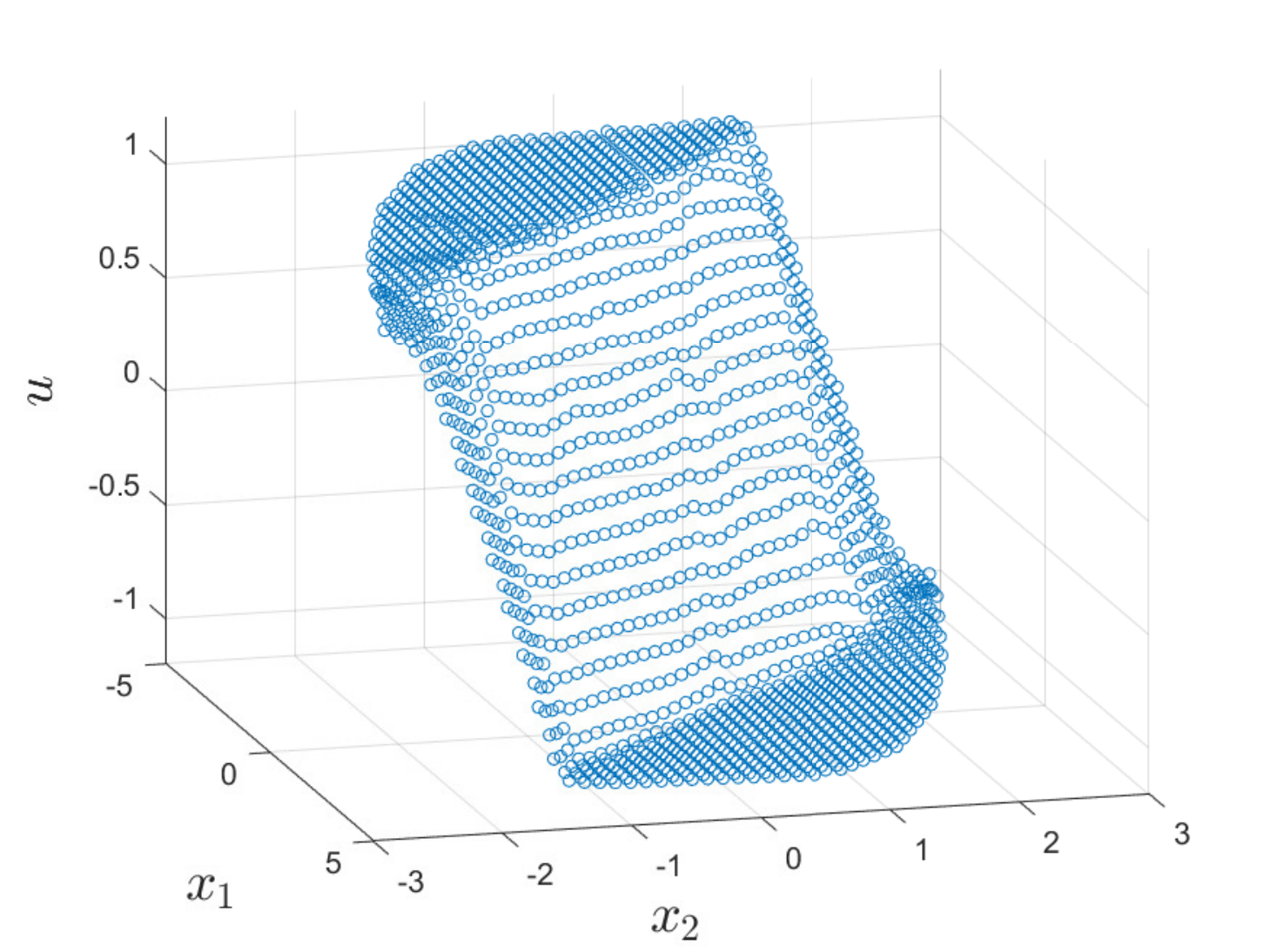}
		\caption{Neural network controller}
		\label{fig:NN_controller}
	\end{subfigure} \hfil
	\caption[Caption for LOF]{The explicit MPC controller (left) is approximated by a neural network controller (right). The ROA of the MPC controller $\mathcal{X}_0$ (middle) is used as the ROI reference to guide the search for Lyapunov functions.\footnotemark}
	\label{fig:MPC}
\end{figure}
\footnotetext{Fig.~\ref{fig:MPC_controller} and Fig.~\ref{fig:MPC_ROA} are generated through the MPT3 toolbox~\cite{MPT3} in MATLAB R2019b.}
\subsubsection{Estimate ROA through Quadratic/PWQ Lyapunov functions}
For the closed-loop system of the double integrator interconnected with $\pi(x)$, we first apply the ACCPM to search for a quadratic Lyapunov function $V(x;P) = x^\top P x$ with $P \in \mathbb{S}^2_{++}$. Since the set $\mathcal{X}_0$ (Fig.~\ref{fig:MPC_ROA}) is the ROA of the MPC controller, the ROI is chosen as $\mathcal{X} = \gamma \mathcal{X}_0 := \{x \vert F_{\mathcal{X}_0} x \leq \gamma h_{\mathcal{X}_0}\}$ with a scaling variable $0 <\gamma\leq 1$ to guide our search of Lyapunov functions. Since the estimate of ROA $\tilde{O}$ is always contained in the ROI, we want $\mathcal{X}$ to be as large as possible as long as the ACCPM can find a Lyapunov function.

In the ACCPM implementation, the learner solves the convex program~\eqref{eq:ac_optimization} through CVXPY~\cite{diamond2016cvxpy} with MOSEK~\cite{mosek} as the solver, and the verifier solves the MIQP~\eqref{eq:pwa_MIQP} through Gurobi~\cite{gurobi}. We set $\epsilon = 0.0172$ in the MIQP~\eqref{eq:pwa_MIQP} and denote $B_\epsilon$ the $\ell_\infty$-norm ball centered at the origin with radius $\epsilon$. Inside $B_\epsilon$, the closed-loop dynamics is given by 
\begin{equation*}
x_+ = A_{cl} x= \begin{bmatrix}
0.50355752 &  0.02697626\\ -0.29822124 &  0.56348813
\end{bmatrix}x
\end{equation*}
Since $A_{cl}$ has eigenvalues $0.53352282  \pm 0.08453977 i$, with norm less than $1$, the closed-loop system is asymptotically stable inside $B_\epsilon$. The ACCPM starts with an empty initial sample set $\mathcal{S}_0 = \emptyset$. Through bisection, the largest scaling variable is given by $\gamma = 0.89$, i.e., we can find a valid quadratic Lyapunov function through the ACCPM with $\mathcal{X} = 0.89 \mathcal{X}_0$. After $10$ iterations, the ACCPM terminated with total solver time $9.788$ seconds and found a valid quadratic Lyapunov function $V(x) = x^\top P x$ with 
\begin{equation} \label{eq:DI_P_value}
P = \begin{bmatrix}
0.14240707 & 0.02797589 \\0.02797589 & 0.78732241
\end{bmatrix}
\end{equation}
together with an estimate of ROA shown in Fig.~\ref{fig:DI_quad_ROA}. In running the ACCPM, we show the optimal values of the MIQP~\eqref{eq:pwa_MIQP} and the counterexamples found by the verifier in Fig.~\ref{fig:DI_history}. At termination, the optimal solution of the MIQP~\eqref{eq:pwa_MIQP} is $x_* = (-0.0172,  0.0048)$ (red star in Fig.~\ref{fig:DI_state_history}) which lies on the boundary of $B_{\epsilon}$ with optimal value $p^* = -7.083 \times 10^{-7} < 0$. This certifies that the terminating parameter $P$ in~\eqref{eq:DI_P_value} generates a Lyapunov function.  


\begin{figure}[htb!]
	\centering
	\begin{subfigure}{0.49 \textwidth}
		\includegraphics[width = \linewidth]{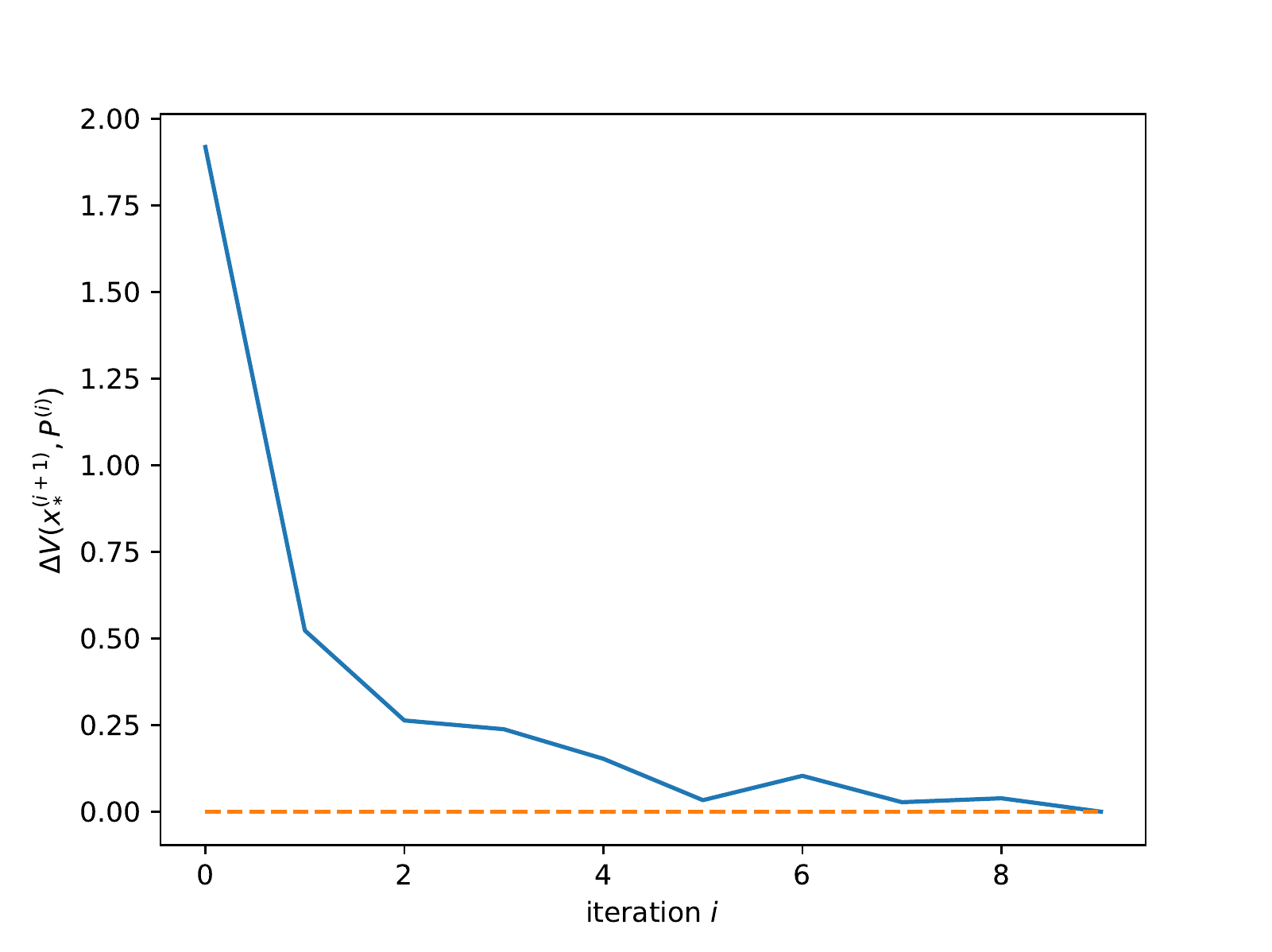}
		\caption{Optimal value of the MIQP~\eqref{eq:pwa_MIQP}}
		\label{fig:DI_value_history}
	\end{subfigure}
	\begin{subfigure}{0.49 \textwidth}
		\includegraphics[width =  \linewidth]{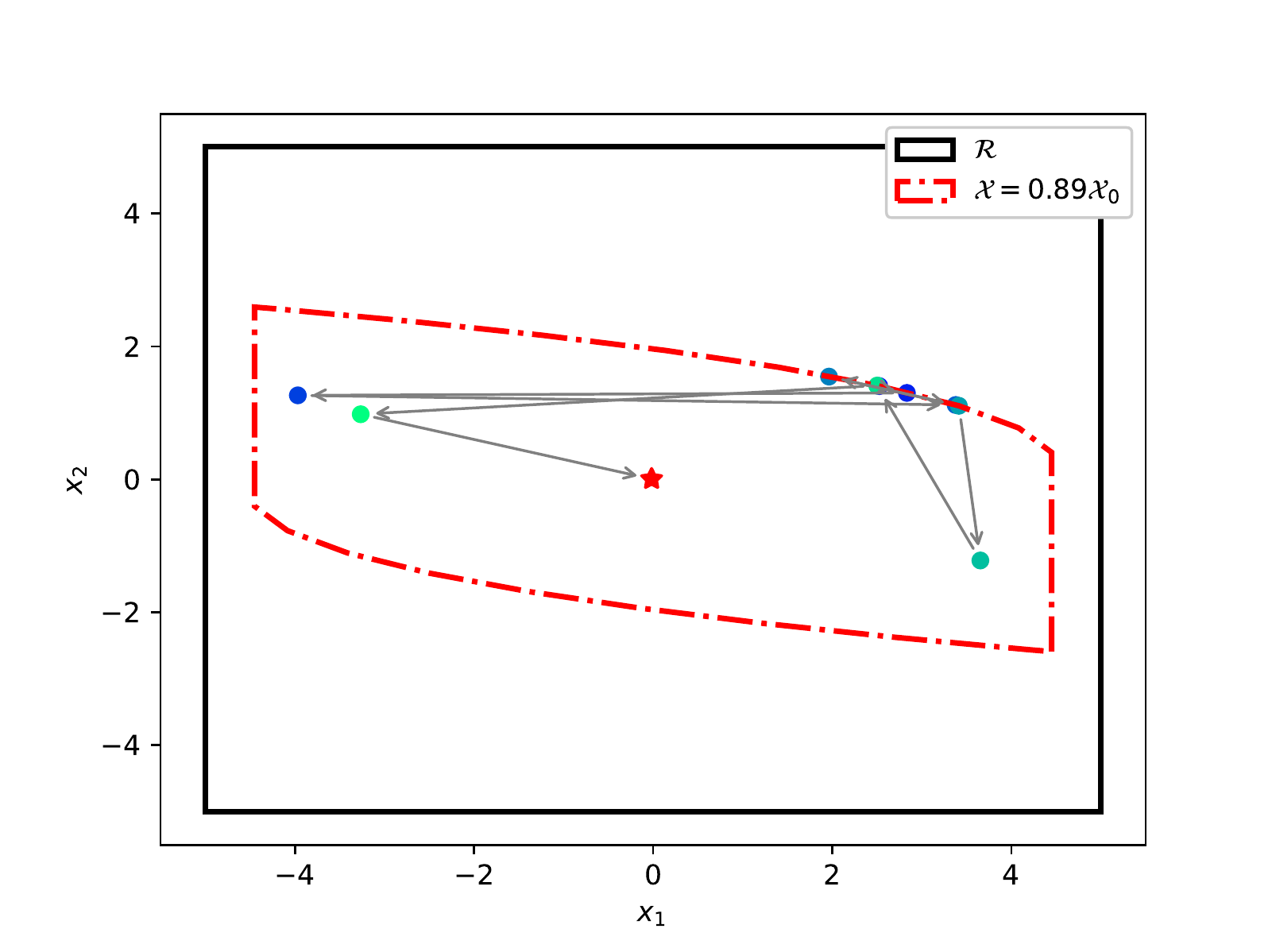}
		\caption{Sequence of counterexamples}
		\label{fig:DI_state_history}
	\end{subfigure} \hfil
	\caption[Caption for LOF]{For the double integrator system with ROI $\mathcal{X} = 0.89 \mathcal{X}_0$ and quadratic Lyapunov function candidates, the ACCPM terminates in $10$ iterations. Fig.~\ref{fig:DI_value_history} plots the optimal value of the MIQP~\eqref{eq:pwa_MIQP} solved by the verifier. Fig.~\ref{fig:DI_state_history} shows the counterexamples found in each iteration with the optimal solution of~\eqref{eq:pwa_MIQP} in the last iteration marked by the red star.} 
	\label{fig:DI_history}
\end{figure}

The same procedure is applied to searching PWQ Lyapunov functions with $P \in \mathbb{S}^{4}$. The ACCPM is initialized with $\mathcal{S}_0 = \emptyset$ and through bisection on $\gamma$ we set $\mathcal{X} = 1.0 \mathcal{X}_0$. In this case, the ACCPM terminates in $5$ iterations with total solver time $51.398$ seconds, and generates a Lyapunov function with
\begin{equation}
P = \begin{bmatrix}
0.27336075 &  0.04068795 & -0.1110683  & 0.09147914 \\
0.04068795 & 0.79551997 & 0.02651292 & 0.15342152 \\
-0.1110683 &  0.02651292 &  0.13499793 & -0.10682033\\
 0.09147914 &  0.15342152 & -0.10682033 & 0.48064250 
\end{bmatrix}.
\end{equation}
The corresponding estimate of ROA given by the PWQ Lyapunov function is shown in Fig.~\ref{fig:DI_pwq_ROA}. Compared with the quadratic Lyapunov function example in Fig.~\ref{fig:DI_quad_ROA}, the application of PWQ Lyapunov functions enlarges the estimate ROA found by the ACCPM.

\begin{figure}[htb!]
	\centering
	\begin{subfigure}{0.49\textwidth}
		\centering
		\includegraphics[width = 0.99 \linewidth]{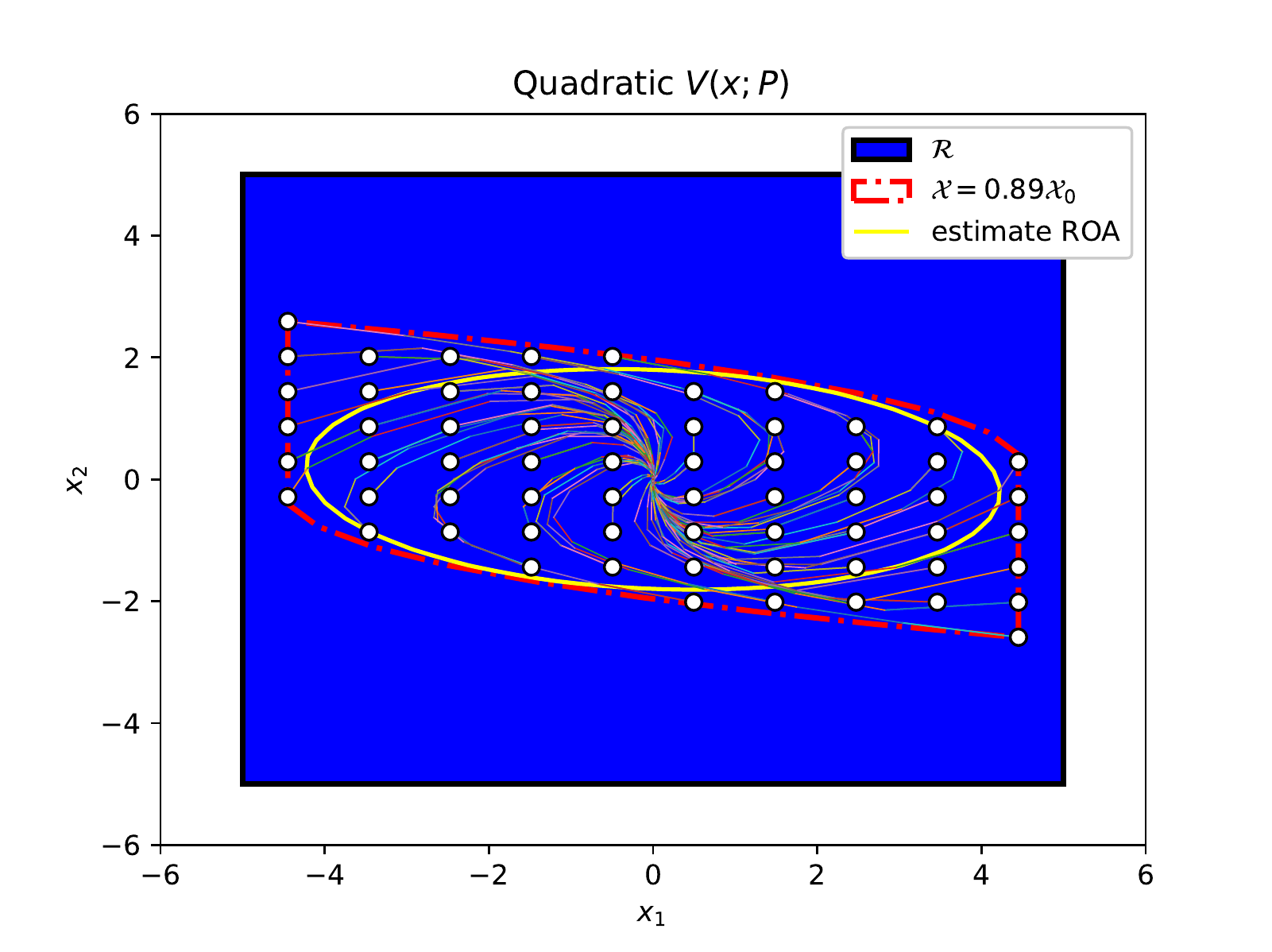}
		\caption{Estimate of ROA by a quadratic Lyapunov function. }
		\label{fig:DI_quad_ROA}
	\end{subfigure} \hfil
	\begin{subfigure}{0.49\textwidth}
		\centering
		\includegraphics[width = 0.99 \linewidth]{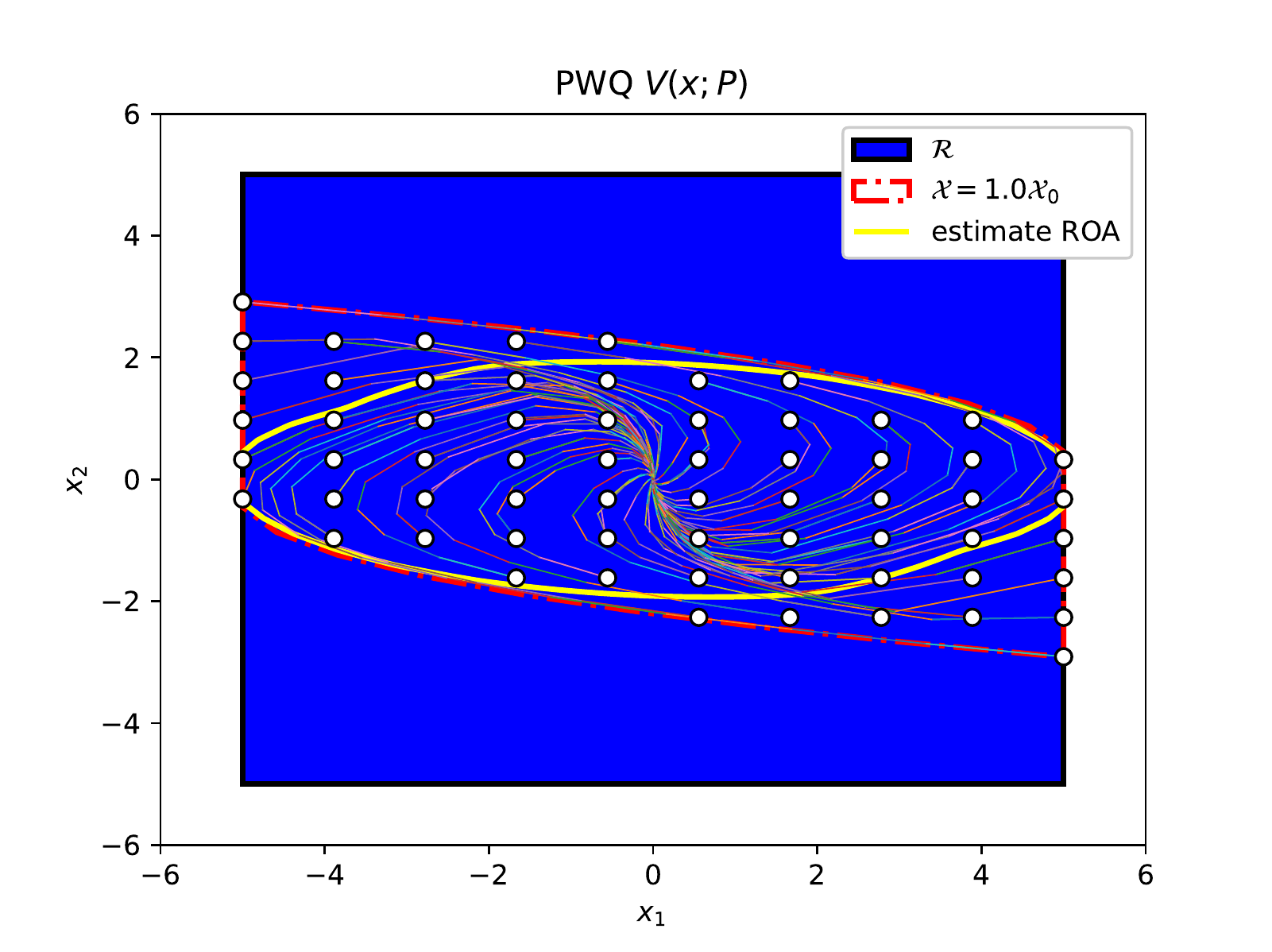}
		\caption{Estimate of ROA by a PWQ Lyapunov function. }
		\label{fig:DI_pwq_ROA}
	\end{subfigure} 
	\caption{Left: with $\mathcal{X}= 0.89 \mathcal{X}_0$, the ACCPM found a quadratic Lyapunov function with an ellipsoidal estimate of ROA (yellow) for $\pi(x)$. Right: with the PWQ Lyapunov function candidate parameterized in~\eqref{eq:PWQ_Lyap}, the ACCPM found a valid candidate with $\mathcal{X}= 1.0 \mathcal{X}_0$ and an estimate ROA (yellow). Simulated closed-loop trajectories with the NN controller are plotted for a grid of initial conditions.}
\end{figure}

\subsubsection{NN controller with projection}
We project the synthesized NN controller $\pi(x)$ to the state-dependent polytope $\Omega(x)$ shown in~\eqref{eq:project_polytope} with $\mathcal{X} = \mathcal{X}_0, \mathcal{U} = U$ since $\mathcal{X}_0$ is a control invariant set under the control input constraint $U$ in~\eqref{eq:DI_dynamics} by construction. By applying the MIQP~\eqref{eq:MIP_proj} as the verifier, the ACCPM found a PWQ Lyapunov function $V(x;P)$ with the initial sample set $\mathcal{S}_0 = \emptyset$ in $5$ iterations and with total solver time $64.470$ seconds. Thus, we certify that $\mathcal{X}_0 \subseteq \mathcal{O}$ when the NN controller with projection is applied. As shown in Fig.~\ref{fig:DI_proj_ROA}, the corresponding estimate of ROA now becomes a polytope and preserves the asymptotic stability guarantee of the MPC controller constructed in Section~\ref{sec:DI_example}.
 
\begin{figure}[htb!]
	\centering
	\includegraphics[width = 0.5 \textwidth]{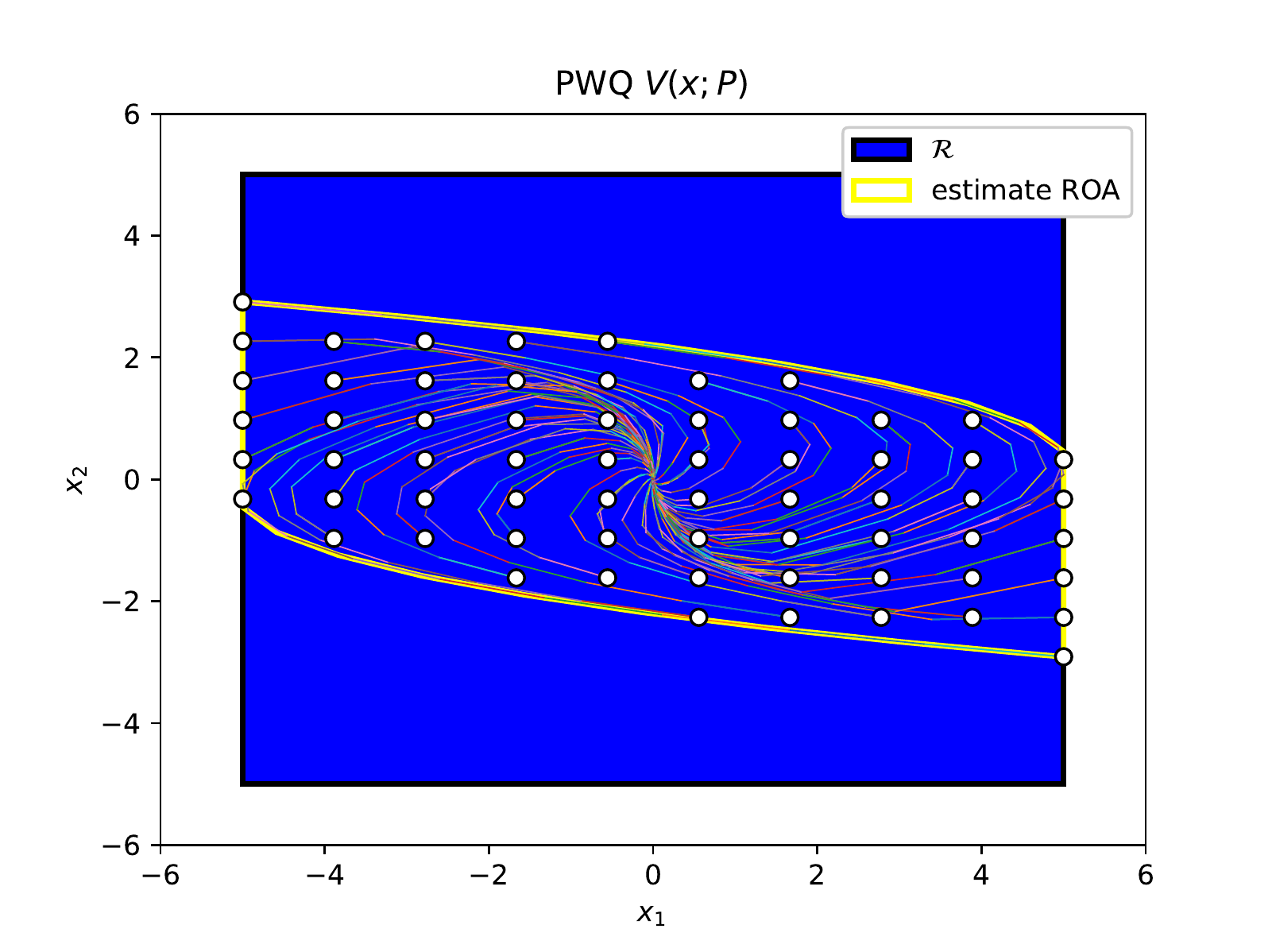}
	\caption{For the NN controller with projection, the polytopic ROI $\mathcal{X} = \mathcal{X}_0$ is an estimate of ROA.}
	\label{fig:DI_proj_ROA}
\end{figure}

\subsection{Inverted pendulum with an elastic wall}
We use a hybrid system example from~\cite{marcucci2017approximate} and consider the inverted pendulum shown in Fig.~\ref{fig:inverted_pendulum} with parameters $m = 1, \ell = 1, g = 10, k = 100, d = 0.1, h = 0.01$. The state is $x = (q, \dot{q})$ which represents the angle $q$ and angular velocity $\dot{q}$ of the pendulum. By linearizing the	 dynamics of the inverted pendulum around $x = 0$, we obtain a hybrid system which has two modes: not in contact with the elastic wall (mode $1$) and in contact with the elastic wall (mode $2$). After discretizing the model using the explicit Euler scheme with a sampling time $h = 0.01$, a PWA model $x_+ = \psi(x, u)$ in the form~\eqref{eq:hybrid_dyn} is obtained with the following parameters
\begin{equation} \label{eq:pendulum_dynamics}
\begin{aligned}
& A_1 = \begin{bmatrix}
1 & 0.01 \\ 0.1 & 1
\end{bmatrix}, B_1 = \begin{bmatrix}
0 \\ 0.01 
\end{bmatrix}, c_1 = \begin{bmatrix}
0 \\ 0
\end{bmatrix}, \mathcal{R}_1 = \{x \vert (-0.2,-1.5) \leq x \leq (0.1, 1.5)\}
\\ 
& A_2 = \begin{bmatrix}
1 & 0.01\\
-0.9 & 1 
\end{bmatrix}, B_2 = \begin{bmatrix}
0 \\ 0.01
\end{bmatrix}, c_2 = \begin{bmatrix}
0 \\ 0.1
\end{bmatrix}, \mathcal{R}_2 = \{x \vert (0.1,-1.5) \leq x \leq (0.2, 1.5)\}
\end{aligned}
\end{equation}

\begin{figure}[htb!]
	\centering
	\begin{subfigure}[t]{0.49 \textwidth}
		\centering
		\includegraphics[width = 0.65 \linewidth]{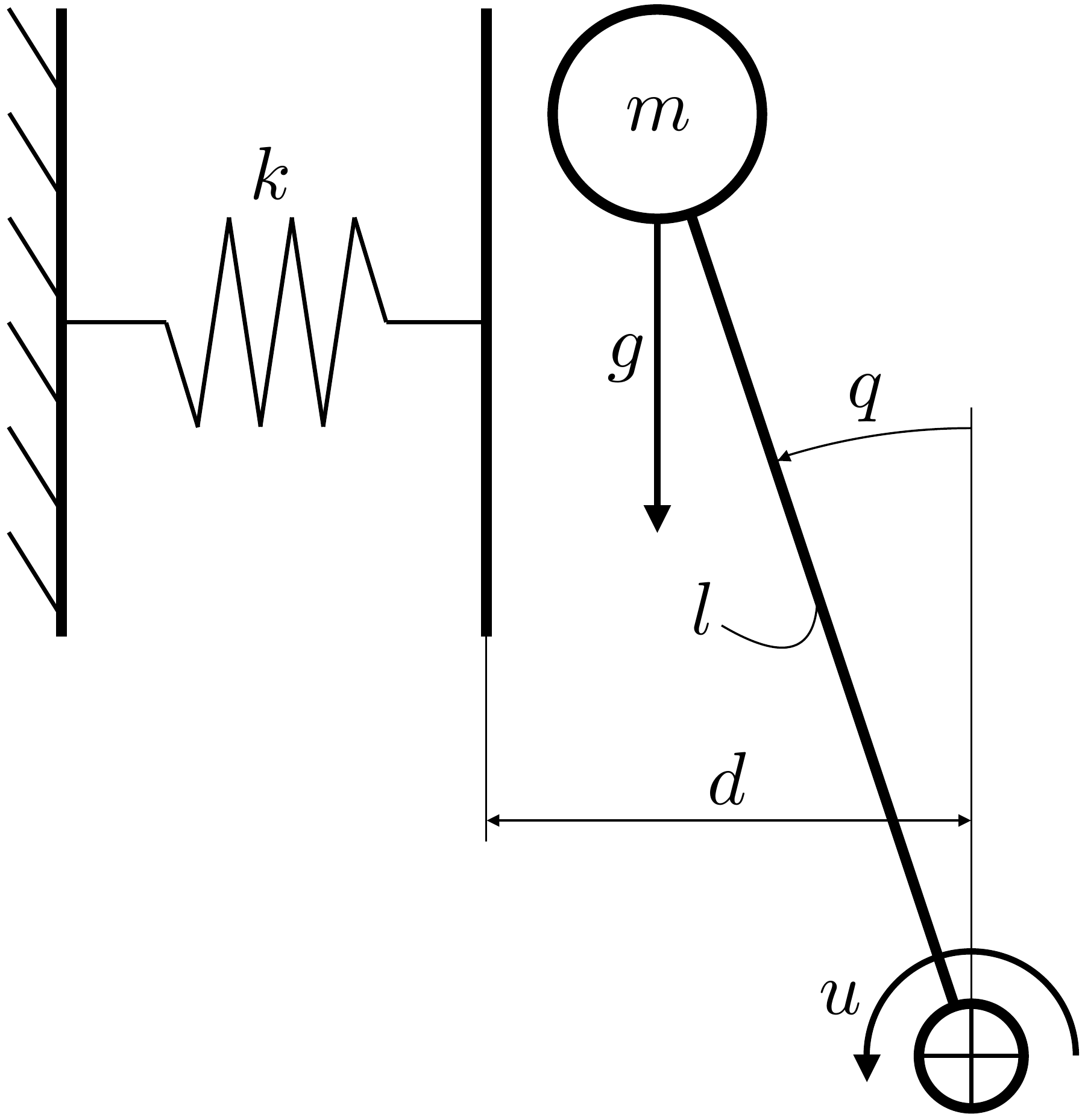}
		\caption{Inverted pendulum with an elastic wall~\cite{marcucci2017approximate}.}
		\label{fig:inverted_pendulum}
	\end{subfigure} \hfil
	\begin{subfigure}[t]{0.49 \textwidth}
		\centering
		\includegraphics[width =  \linewidth]{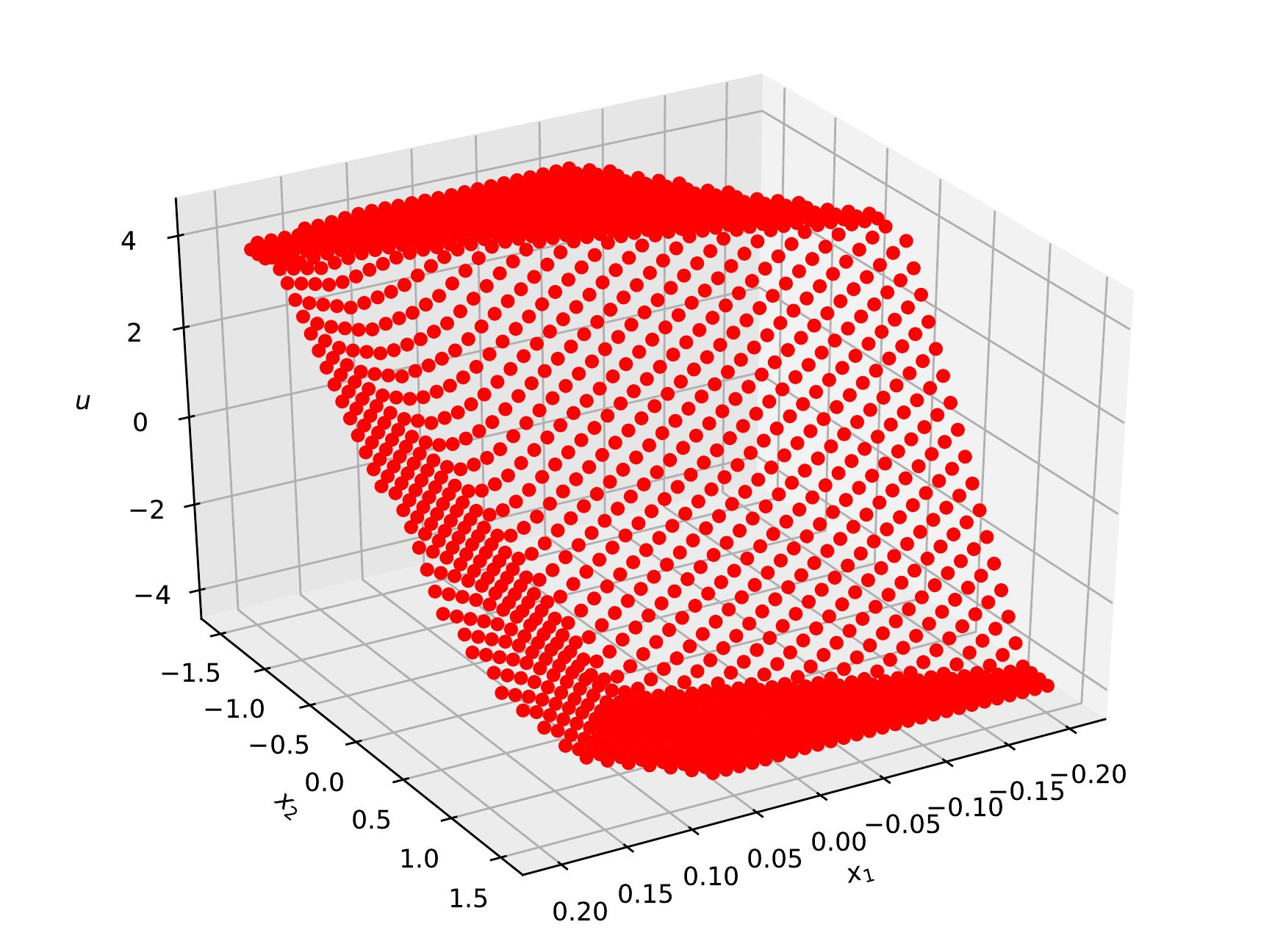}
		\caption{NN controller for the inverted pendulum system.}
		\label{fig:IP_nn_controller}
	\end{subfigure}
	\caption{A ReLU neural network (right) that approximates a hybrid MPC controller is applied on the inverted pendulum system with an elastic wall (left).}
\end{figure}

We then synthesize a hybrid MPC controller $\pi_{MPC}(x)$ for the PWA system through disjunctive programming~\cite{marcucci2019mixed} where the control input constraints are given by $u \in U = \{ u \in \mathbb{R} \vert -4 \leq u \leq 4 \}$ and the horizon of MPC is set as $T = 10$. The stage and terminal costs are given by 
\begin{equation*}
q(x_t, u_t) = x_t^\top Q x_t + u_t^\top R u_t, \quad p(x_T) = x_T^\top Q_T x_T
\end{equation*} 
where $Q = I, R = 1$, and $Q_T = \text{DARE}(A_1, B_1, Q, R)$ is the solution to the discrete algebraic Riccati equation defined by $(A_1, B_1, Q, R)$. We evaluate $\pi_{MPC}(x)$ on a uniform $40 \times 40$ grid samples from the state space and let the reference ROI $\mathcal{X}_0$ be the convex hull of all the feasible state samples. Then the ROI $\mathcal{X} = \gamma \mathcal{X}_0$ with $0 < \gamma \leq 1$ is applied to guide the search for an estimate of ROA.

A total number of $1354$ feasible samples of state and control input pairs $(x, \pi_{MPC}(x))$ are generated to train a ReLU neural network $\pi(x)$ in Keras to approximate the MPC controller. The NN has $2$ hidden layers with $20$ neurons in each layer and its output layer bias term is modified after training to guarantee $\pi(0) = 0$. The NN controller is shown in Fig.~\ref{fig:IP_nn_controller}. For the NN controller, we set $\epsilon = 0.0158$ and verify that the linear closed-loop dynamics inside $B_\epsilon$ is asymptotically stable.

With empty initial sample set $\mathcal{S}_0$, the ACCPM is run with both quadratic and PWQ Lyapunov function candidates in order to find a large estimate of ROA. With quadratic Lyapunov function candidates, the largest ROI is given by $\mathcal{X} = 0.81 \mathcal{X}_0$ through bisection on $\gamma$. The ACCPM terminates in $10$ iterations with total solver time $9.928$ seconds. The corresponding estimate of ROA is shown in Fig.~\ref{fig:IP_quad_ROA}. With PWQ Lyapunov function candidates, the largest ROI is given by $\mathcal{X} = 1.0 \mathcal{X}_0$ in which case the ACCPM terminates in $9$ iterations with total solver time $124.210$ seconds. The estimate of ROA obtained by the found PWQ Lyapunov function is shown in Fig.~\ref{fig:IP_PWQ_ROA}. It is observed that the PWQ Lyapunov function is less conservative compared with the quadratic one. In addition to the theoretical guarantees provided by the ACCPM, the validity of the estimate of ROA is also shown by the simulated closed-loop trajectories in Fig.~\ref{fig:IP_ROA}.

\begin{figure}
	\centering
	\begin{subfigure}{0.49\textwidth}
		\centering
		\includegraphics[width = 0.99 \linewidth]{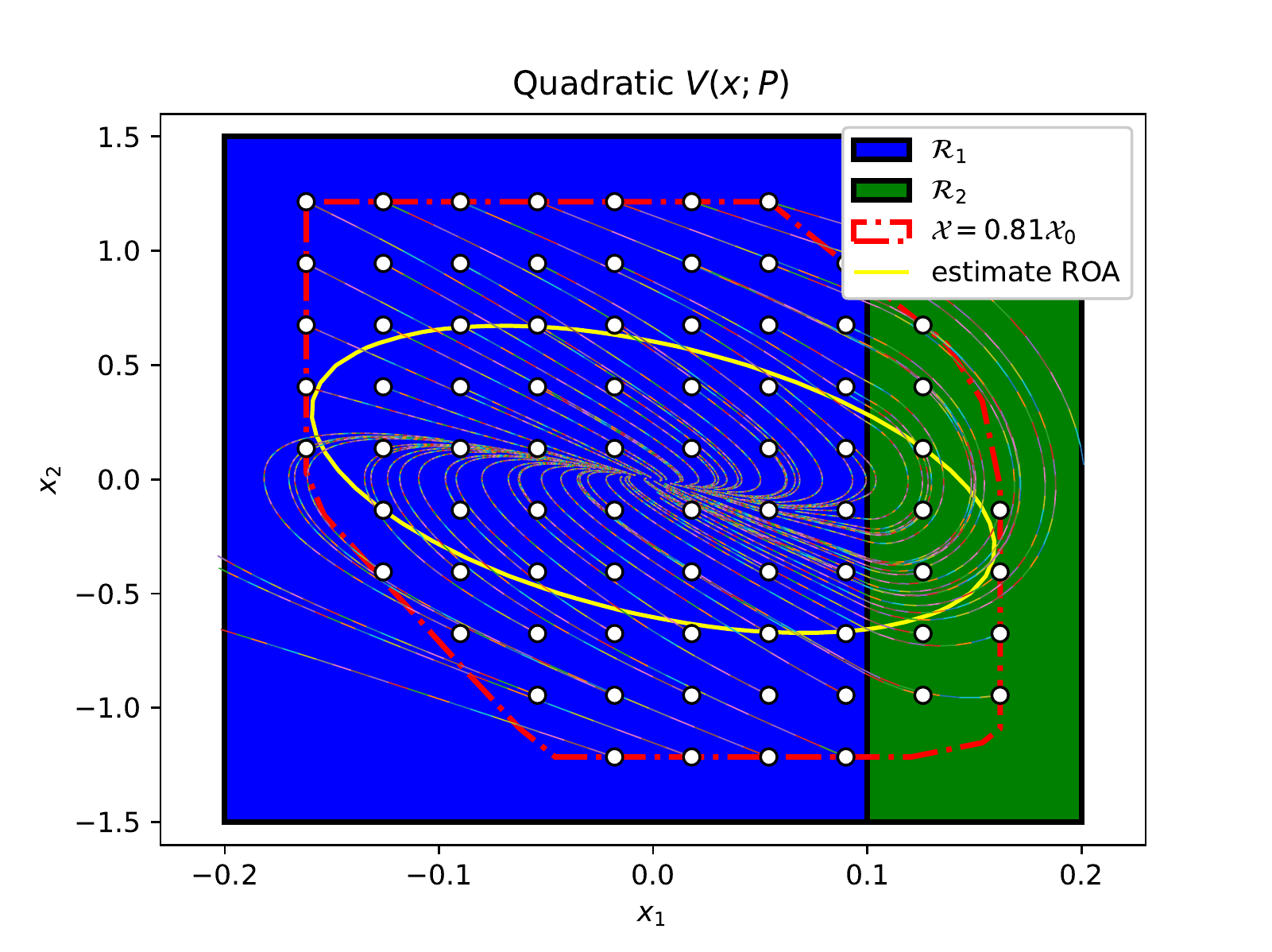}
		\caption{Estimate of ROA by a quadratic Lyapunov function. }
		\label{fig:IP_quad_ROA}
	\end{subfigure} \hfil
	\begin{subfigure}{0.49\textwidth}
		\centering
		\includegraphics[width = 0.99 \linewidth]{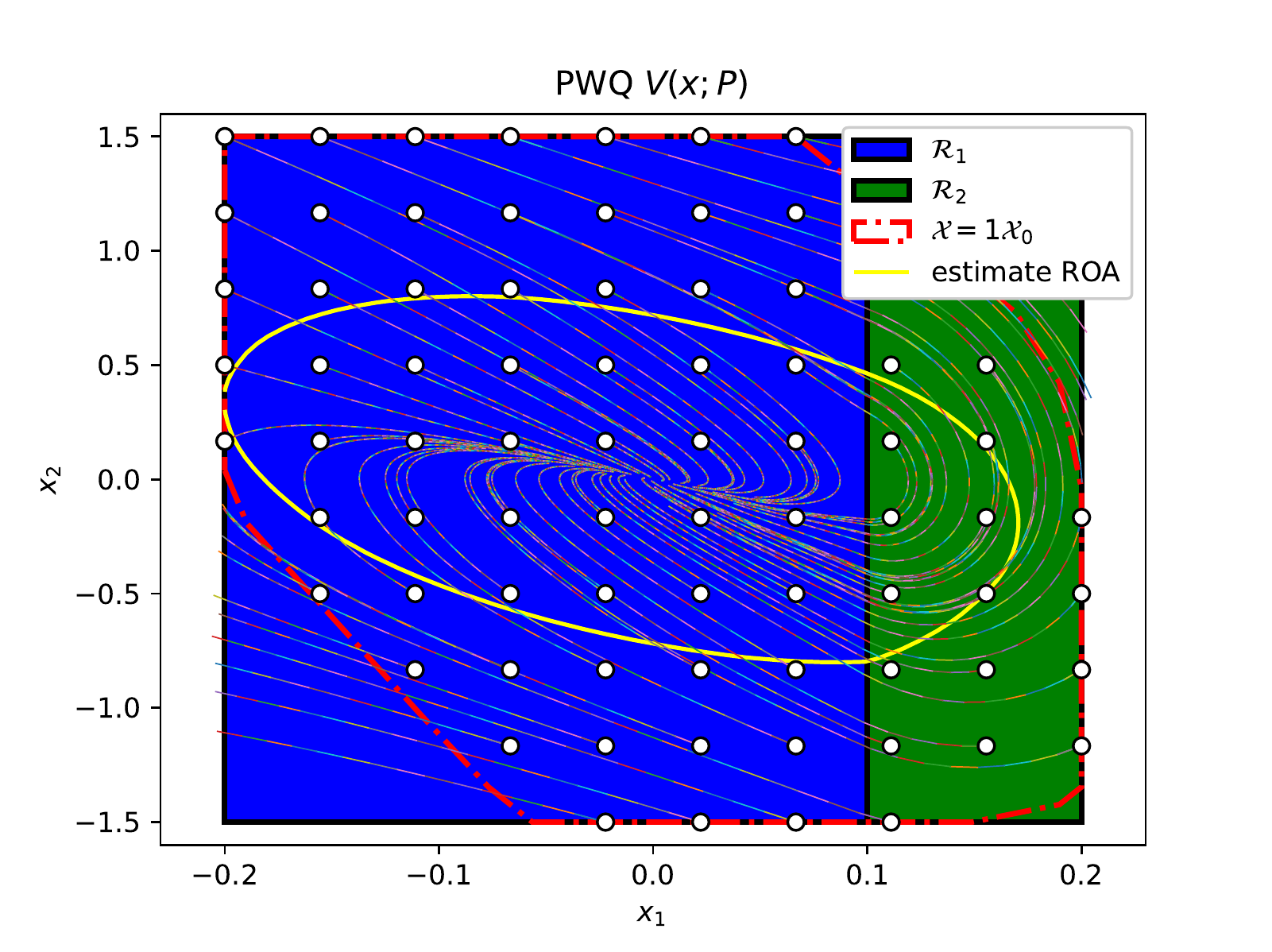}
		\caption{Estimate of ROA by a PWQ Lyapunov function. }
		\label{fig:IP_PWQ_ROA}
	\end{subfigure}
	\caption{Estimates of ROA found by the ACCPM with quadratic (left) and PWQ (right) Lyapunov functions. The partitions $\mathcal{R}_1$ and $\mathcal{R}_2$ are marked blue and green, respectively. Simulated closed-loop trajectories with the NN controller are plotted for a grid of initial conditions.}
	\label{fig:IP_ROA}
\end{figure}


\section{Conclusion}
\label{sec:conclusion}
In this paper, we have proposed an iterative algorithm that learns quadratic and piecewise quadratic Lyapunov functions for piecewise affine systems with ReLU neural network controllers. The proposed algorithm is composed of a learner and a verifier. The learner uses a cutting-plane strategy to propose Lyapunov function candidates from a set of sampled states of the closed-loop system, while the verifier either certifies the validity of the proposed Lyapunov function or rejects it with a counterexample to be accounted for by the learner in the next round. We provide finite-step termination guarantee for the overall algorithm when the set of Lyapunov functions is full-dimensional in the parameter space. Future work includes extending the proposed algorithm to stability analysis of neural-network-controlled uncertain systems. 

\small
\bibliography{reference}
\bibliographystyle{ieeetr}

\end{document}